\documentclass{amsart}
\usepackage[T1]{fontenc}
\usepackage{setspace}
\usepackage{amsmath,amsfonts,amscd,amssymb}
\usepackage{graphicx}
\usepackage[all]{xy}

\newcommand{\bir}{\dashrightarrow}

\newcommand{\mS}{\mathcal{S}}

\newcommand{\mJ}{\mathcal{J}}

\newcommand{\mC}{\mathcal{C}}
\newcommand{\mK}{\mathcal{K}}
\newcommand{\mH}{\mathcal{H}}

\newcommand{\mG}{\mathcal{G}}

\newcommand{\mF}{\mathcal{F}}

\newcommand{\mD}{\mathcal{D}}

\newcommand{\DD}{\mathbb{D}}
\newcommand{\ZZ}{\mathbb{Z}}

\newcommand{\CC}{\mathbb{C}}

\newcommand{\RR}{\mathbb{R}}
\newcommand{\PP}[1]{\mathbb{CP}^{#1}}
\newcommand{\RP}[1]{\mathbb{RP}^{#1}}
\newcommand{\id}{\mbox{id}}

\newcommand{\Orb}{\mbox{Orb}}
\newcommand{\Map}{\mbox{Map}}
\newcommand{\Coll}{\mbox{Col}}
\newcommand{\Conf}{\mbox{Conf}}
\newcommand{\Diff}{\mbox{Diff}}
\newcommand{\Symp}{\mbox{Symp}}

\newcommand{\rot}{\mbox{rot}}
\newcommand{\Endo}{\mbox{End}}

\newcommand{\Br}{\mbox{Br}}

\newcommand{\ev}{\mbox{ev}}

\newcommand{\Stab}{\mbox{Stab}}
\newcounter{unthm}
\newcounter{unlma}
\newtheorem{thmun}[unthm]{Theorem}
\newtheorem{lmaun}[unlma]{Lemma}
\newtheorem{thm}{Theorem}
\newtheorem{dfn}[thm]{Definition}

\newtheorem{lma}[thm]{Lemma}
\newtheorem{prp}[thm]{Proposition}

\newtheorem{rmk}[thm]{Remark}

\newtheorem{cnd}{Condition}
\title[Symplectic mapping class groups]%
{Symplectic mapping class groups of some Stein and rational surfaces}
\author{J. D. Evans}
\address{DPMMS, Centre for Mathematical Sciences, Wilberforce Road, Cambridge,
CB3 0WA}
\email{j.evans@dpmms.cam.ac.uk}
\begin{document}
\begin{abstract}
In this paper we compute the homotopy groups of the symplectomorphism groups of
the 3-, 4- and 5-point blow-ups of the projective plane (considered as monotone
symplectic Del Pezzo surfaces). Along the way, we need to compute the homotopy
groups of the compactly supported symplectomorphism groups of the cotangent
bundle of $\RP{2}$ and of $\CC^*\times\CC$. We also make progress in the case of the
$A_n$-Milnor fibres: here we can show that the (compactly supported)
Hamiltonian group is contractible and that the symplectic mapping class group
embeds in the braid group on $n$-strands.
\end{abstract}
\maketitle
\tableofcontents

\section{Introduction}

A central object in symplectic topology is the group $\Symp(X,\omega)$ of (compactly supported) diffeomorphisms of a symplectic manifold $X$ which preserve the symplectic form $\omega$. This paper studies the weak homotopy type of $\Symp(X,\omega)$ for some simple symplectic 4-manifolds. We begin by reviewing what is known about the topology of symplectomorphism groups.

\subsection{Review}

Gromov \cite{Gro85} gave a beautiful proof of the following theorem:

\begin{thm}
The symplectomorphism group of $S^2\times S^2$ with its product symplectic form
$\omega\oplus\omega$ deformation retracts onto $SO(3)\times SO(3)\rtimes\ZZ/2$.
\end{thm}

This is just one corollary of his theory of pseudholomorphic
curves in symplectic manifolds. The tractability of $\Symp(S^2\times S^2)$ derives from the following existence theorem for pseudoholomorphic foliations:

\begin{thm}\label{folnthm}
If $J$ is an almost complex structure on $X=S^2\times S^2$ which is compatible
with the product symplectic form $\omega$ then there exist $J$-holomorphic
foliations $\mF_1$ and $\mF_2$ of $X$ whose leaves are spheres in the homology
classes $S^2\times\{pt\}$ and $\{pt\}\times S^2$ respectively.
\end{thm}

Since $J$-holomorphic curves intersect positively, a leaf of $\mF_1$ and a leaf
of $\mF_2$ intersect transversely in a single point. This allows one to
construct diffeomorphisms conjugating the foliation pairs coming from different
almost complex structures. Modifications of these arguments allow one to deduce
that:

\begin{itemize}
\item $\Symp(\PP{2},\omega)$ is homotopy equivalent to $\mathbb{P}U(3)$, where
$\omega$ is the Fubini-Study form (due to Gromov, \cite{Gro85}),
\item $\Symp(\DD_1,\omega)$ is homotopy equivalent to $U(2)$, where $\DD_1$ is
the 1-point blow-up of $\PP{2}$ and $\omega$ is the anticanonical K\"{a}hler
form, (due to Gromov, \cite{Gro85}),
\item $\Symp(\DD_2,\omega)$ is weakly homotopy equivalent to $T^2\rtimes\ZZ/2$,
where $\DD_2$ is the 2-point blow-up of $\PP{2}$ and $\omega$ is the
anticanonical K\"{a}hler form (due to Lalonde-Pinsonnault, \cite{LP04}),
\item $\Symp_c(\CC^2)$ is contractible, where $\omega$ is the product
symplectic form (due to Gromov, \cite{Gro85}),
\item $\Symp_c(T^*S^2)$ is weakly homotopy equivalent to $\ZZ$, where $\omega$ is the canonical symplectic form (by argument due to Seidel, see \cite{Sei98}).
\end{itemize}

Here $\Symp_c$ denotes the group of compactly supported symplectomorphisms.
These arguments work because $\CC^2$ and $T^*S^2$ embed into $S^2\times S^2$
and inherit foliations by pseudoholomorphic discs from the foliations of
theorem \ref{folnthm}.

One might hope that if one could calculate $\pi_i\left(\Symp_c(A))\}_{i\in\mathbb{N}}\right)$ for some other affine
varieties $A$ then one might use Gromov's theory of pseudoholomorphic curves in
symplectic 4-manifolds to deduce something about the topology of $\Symp(X)$
where $X$ is a projective variety containing $A$ as a Zariski open set. This
paper accomplishes this in a small number of cases.

Note that there are also some results on weak contractibility of $\Symp_c$ for
symplectic disc bundles over $S^2$, from the work of Richard Hind \cite{Hin03}
and on $\Symp_c(\CC^*\times\CC^*)$ from the work of Chris Wendl \cite{Wen08}.
These results are based on arguments from symplectic field theory. There is
also a lot of work about the topology of $\Symp(S^2\times S^2,\omega)$ where
$\omega$ is a non-monotone symplectic form (for example, see \cite{Abr98}). In
this paper we are almost exclusively concerned with monotone symplectic forms
(except in the final section), so we do not review this literature. It should
be noted, however, that the techniques of \cite{Abr98} have heavily inspired
the techniques of this paper and this will be immediately evident to readers
who are familiar with that work.

\subsection{New results}

The aim of this paper is to add to the list of symplectic 4-manifolds whose
symplectomorphism group we understand (up to weak homotopy equivalence). In
particular, we prove the following theorems:

\begin{thm}\label{sympgrps}
Let $\DD_n$ denote the blow-up of $\PP{2}$ at $n\leq 5$ points in general
position and let $\omega$ denote its anticanonical K\"{a}hler form. In each
case let $\Symp_0(\DD_n)$ denote the group of symplectomorphisms of
$(\DD_n,\omega)$ which act trivially on $H_*(\DD_n,\ZZ)$.
\begin{itemize}
\item $\Symp_0(\DD_3)$ is weakly homotopy equivalent to $T^2$,
\item $\Symp_0(\DD_4)$ is weakly contractible,
\item $\Symp_0(\DD_5)$ is weakly homotopy equivalent to $\Diff^+(S^2,5)$, the
group of orientation-preserving diffeomorphisms of $S^2$ preserving five
points.
\end{itemize}
\end{thm}

The author has since learned that these results on monotone symplectic Del Pezzo surfaces have also been obtained independently by Martin Pinsonnault \cite{Pin}.

\begin{thm}\label{milnorsymp}
Let $W$ be the $A_n$-Milnor fibre, the affine variety given by the equation
\[x^2+y^2+z^n=1,\]
and let $\omega$ be the K\"{a}hler form on $W$ induced from the ambient
$\CC^3$. Then the group of compactly supported symplectomorphisms of
$(W,\omega)$ is weakly homotopy equivalent to its group of components. This
group of components injects homomorphically into the braid group $\Br_n$ of
$n$-strands on the disc.
\end{thm}

En route to proving these theorems, we prove and use:

\begin{thm}
Let $\omega$ be the canonical symplectic form on the cotangent bundle of
$\RP{2}$. The group of compactly supported symplectomorphisms
$\Symp_c(T^*\RP{2})$ is weakly homotopy equivalent to $\ZZ$.
\end{thm}

\begin{thm}
Let $\omega$ be the product symplectic form on $\CC^*\times\CC$. The group of
compactly supported symplectomorphisms $\Symp_c(\CC^*\times\CC)$ is weakly
contractible.
\end{thm}

\subsection{Comments}

Theorem \ref{milnorsymp} is not as strong as we would like: the injection is
probably also a surjection. I have been unable to prove this for technical
reasons. However, we can still deduce something more about the symplectic
mapping class group $\pi_0(\Symp_c(W))$ in the case $n\geq 4$. Khovanov and
Seidel demonstrated \cite{KS01} using Fukaya categories that the braid group
$\Br_n$ always injects into $\pi_0(\Symp_c(W))$. The composition of their
injection with ours is therefore a homomorphic injection $K:\Br_n\rightarrow
\Br_n$. While the braid group is not co-Hopfian, it is known \cite{BM06} that,
when $n\geq 4$, all such injections are of the form
\[\sigma_i\mapsto z^{\ell}\left(h^{-1}\sigma_i h\right)^{\pm 1}\]
where $\sigma_i$ are the usual generators, $h$ is a homeomorphism of the marked
disc, $z\in Z(\Br_n)$ is a full-twist and $\ell\in\ZZ$. In particular, we know
that the subgroup $K(\Br_n)$ is of finite index and that all intermediate
subgroups between $K(\Br_n)$ and $\Br_n$ are isomorphic to $\Br_n$, in
particular the group $\pi_0(\Symp_c(W))$ is abstractly isomorphic to $\Br_n$.

Seidel \cite{Sei08} has also obtained results in the direction of
theorem \ref{sympgrps}, specifically exhibiting a subgroup of
$\pi_0(\Symp_0(\DD_5))$ isomorphic to $\pi_0(\Diff^+(S^2,5))$. These results
were the starting point for this paper and will be reviewed in section
\ref{delpres}.

Note that the weak homotopy equivalences proved in the theorems actually suffices to prove homotopy equivalence since, by \cite{MS04} Remark 9.5.6, these topological groups have the homotopy type of a countable CW complex.

\subsection{Philosophy}

The methods of this paper are very classical, along the lines of \cite{Abr98} and \cite{LP04}; the only `hard' input is from Gromov's theory of pseudoholomorphic curves. In each case we will identify a divisor in our symplectic manifold (or in a compactification of it). The divisor will only have spherical components. The idea is to relate the topology of the symplectomorphism group to the topology of the `symplectomorphism group' of the divisor, i.e. the group of symplectomorphisms of the components which fix intersection points between different components. For instance, as observed by Seidel, the divisor we use in $\DD_5$ contains a component with five intersection points, leading to the pure braid group on five strands. The difficulty is in working out the precise relationship of this to the symplectomorphism group and showing there is no topology invisible to this divisor. This involves examining a number of long exact sequences of homotopy groups coming from fibrations and understanding the topology of the compactly supported symplectomorphism group of the complement of the divisor.

One might ask why we stopped at $\DD_5$. The reason is that the complement of a conic in $\PP{2}$ (symplectically $T^*\RP{2}$) occurs naturally as a Zariski open subset here. $\DD_6$ is a cubic surface, so one might hope to understand its symplectomorphism group if one could understand symplectomorphisms of the affine cubic surface. However, this does not admit a foliation by rational holomorphic curves, which is ultimately what allows us to understand $\Symp_c(T^*\RP{2})$.

\subsection{Overview}

Section \ref{stein} reviews the properties of non-compact symplectic manifolds and their (compactly supported) symplectomorphism groups relevant for later arguments. The aim of this section is proving Proposition \ref{twostein}, which shows that if two (possible non-complete) symplectic manifolds arise from different plurisubharmonic functions on the same complex manifold then their compactly supported symplectomorphism groups are weakly homotopy equivalent. This will be useful later as we will be able to identify the symplectomorphism groups of complements of ample divisors by identifying the complement up to \emph{biholomorphism}.

Section \ref{rp2} performs the computation of $\pi_k(\Symp_c(T^*\RP{2}))$ and can actually be read independently of the rest of the paper. It closely follows \cite{Sei98}.

In section \ref{cstarcsect} we carry out the computation in detail for $\CC^*\times\CC$. The only real difficulty is in proving weak contractibility of a space of symplectic spheres (playing the role of the divisor mentioned above). This proof is postponed to the end of the section and one crucial lemma is proved in the appendix.

We proceed to calculate the weak homotopy type of the symplectomorphism groups of $\DD_3$, $\DD_4$ and $\DD_5$ in section \ref{delp} and of the $A_n$-Milnor fibre in section \ref{milnorsect}. Many of the proofs are similar to the case $\CC^*\times\CC$, so details are sketched noting where technicalities arise.

\subsection{Acknowledgements}

During this work I was funded by a Faulkes Studentship in Geometry. This work
will form part of my PhD thesis and it gives me pleasure to
acknowledge the many stimulating conversations with my PhD supervisor, Ivan
Smith, which have led to this paper. I would also like to thank Jack Waldron,
Mark McLean, Chris Wendl, Jarek K\k{e}dra, Gabriel Paternain, Dusa McDuff and the paper's referee for useful comments, encouragement and corrections. Particular
thanks go to Jack Button for his comments and suggestions about the co-Hopfian
properties of groups and to Chris Wendl for explaining steps in his paper
\cite{Wen08}. The ideas and arguments in this paper are heavily inspired by
sections of Seidel's marvellous survey \cite{Sei08} and Abreu's beautiful paper
\cite{Abr98}. Since preparing this paper, the author has learned that the results on monotone
symplectic Del Pezzo surfaces have also been obtained independently by Martin Pinsonnault \cite{Pin}.

\section{Stein manifolds}\label{stein}

\subsection{Stein manifolds and Liouville flows}

\begin{dfn}[Plurisubharmonic functions, Stein manifolds]
A function $\phi:W\rightarrow[0,\infty)$ on a complex manifold $(W,J)$ is
called \emph{plurisubharmonic} if it is proper and if $\omega=-d(d\phi\circ J)$
is a positive $(1,1)$-form. A triple $(W,J,\phi)$ consisting of a complex
manifold $(W,J)$ and a plurisubharmonic function $\phi$ on $(W,J)$ is called a
\emph{Stein manifold}. A sublevel set $\phi^{-1}[0,c]$ of a Stein manifold
$(W,J,\phi)$ is called a Stein domain.
\end{dfn}

Given a Stein manifold $(W,J,\phi)$ one can associate to it:
\begin{itemize}
\item a 1-form $\theta=-d\phi\circ J$ which is a primitive for the positive
$(1,1)$-form $\omega$,
\item a vector field $Z$ which is $\omega$-dual to $\theta$.
\end{itemize}
\noindent We call $\theta$ the Liouville form and $Z$ the Liouville vector
field. Let us abstract the notions of Liouville form and vector field further:

\begin{dfn}
A \emph{Liouville manifold} is a triple $(W,\theta,\phi)$ consisting of a
smooth manifold $W$, a (Liouville) 1-form $\theta$ on $W$ and a proper function
$\phi:W\rightarrow[0,\infty)$ such that
\begin{itemize}
\item $\omega=d\theta$ is a symplectic form,
\item there is a monotonic sequence $c_i\rightarrow\infty$ such that on the
level sets $\phi^{-1}(c_i)$ the Liouville vector field $Z$ which is
$\omega$-dual to $\theta$, satisfies $d\phi(Z)>0$ everywhere on
$\phi^{-1}(c_i)$.
\end{itemize}
\noindent $(W,\theta,\phi)$ is said to have \emph{finite-type} if there is a
$k>0$ such that $d\phi(Z)>0$ at $x$ for any $x\in\phi^{-1}(k,\infty)$. It is
said to be \emph{complete} if the Liouville vector field is complete. A closed
sublevel set $\phi^{-1}[0,k]$ of a Liouville manifold is called a
\emph{Liouville domain}.
\end{dfn}

Clearly a Stein manifold inherits the structure of a Liouville manifold and a
Stein domain inherits the structure of a Liouville domain.

On a Liouville domain, the negative-time Liouville flow always exists and one
can use it to define an embedding $\Coll:(-\infty,0]\times\partial W\rightarrow
W$ with $\Coll^*\theta=e^r\theta|_{\partial W}$ and $\Coll_*\partial_r=Z$. We
can therefore form the \emph{symplectic completion} of $(W,\theta,\phi)$, which
is a Liouville manifold:

\begin{dfn}[Symplectic completion]
Let $(W,\theta,\phi)$ be a Liouville domain with boundary $M$ and define
$\alpha=\theta|_M$. The \emph{symplectic completion} $(\hat{W},\hat{\theta})$
of $(W,\theta)$ is the manifold $\hat{W}=W\cup_{\Coll}(-\infty,\infty)\times M$
equipped with the 1-form $\hat{\theta}|_{W}=\theta$,
$\hat{\theta}|_{(-\infty,\infty)\times M}=e^r\alpha$. There is an associated
symplectic form $\hat{\omega}=d\hat{\theta}$ and Liouville field
$\hat{Z}|_{W}=Z$, $\hat{Z}|_{(-\infty,\infty)\times W}=\Coll_*\partial_r$ whose
flow exists for all times. We may take $\hat{\phi}$ to be any smooth extension
of $\phi$ which agrees with $r$ outside some compact subset.
\end{dfn}

There is also a standard construction given a (possibly incomplete) finite-type
Stein manifold $(W,J,\phi)$ to obtain a complete finite-type Stein manifold.
Let $h$ denote the function $h(x)=e^x-1$.

\begin{lma}[\cite{BC01}, Lemma 3.1 and \cite{SS05} Lemma 6]
Define $\phi_h:=h\circ\phi$. $(W,J,\phi_h)$ is a complete Stein manifold of
finite-type.
\end{lma}

In fact, $\phi$ and $\phi_h$ have the same critical points. If $\theta$ and $Z$
are the Liouville form and vector field on $(W,J,\phi)$ then we write
$\theta_h$ and $Z_h$ for the corresponding data on $(W,J,\phi_h)$. This
construction is related to the symplectic completion of a sublevel Stein domain
of $\phi$:

\begin{lma}\label{affinecomplete}
Let $k$ be such that $\phi^{-1}[0,k)$ contains all the critical points of
$\phi$ and set $W_k=\phi^{-1}[0,k]$, which inherits the structure of a
Liouville domain. Then $(\hat{W}_k,d\hat{\theta})$ is symplectomorphic to
$(W,d\theta_h)$.
\end{lma}
\begin{proof}
Pick some $k'>k$. Let $h_0$ be a function such that:
\begin{itemize}
\item $h_0(x)=h(x)$ if $x>k'$,
\item $h_0(x)=x$ if $x\leq k$,
\item $h_0'(x)>0$, $h_0''(x)>0$ for all $x\in[0,\infty)$.
\end{itemize}
\noindent Define $h_t=th+(1-t)h_0$. Now consider the functions
$\phi_t=h_t\circ\phi$ on $W$ for $t\in[0,1]$. By Lemma 6 of \cite{SS05} they
are all plurisubharmonic functions making $(W,J,\phi_t)$ into a complete Stein
manifold and they all have the same set of critical points. Let $\theta_t$
denote $d^c\phi_t$. By Lemma 5 of \cite{SS05} there is a smooth family of
diffeomorphisms $f_t:W\rightarrow W$ such that $f_t^*\theta_t=\theta_0+dR_t$
for some compactly-supported function $R: W\times[0,1]\rightarrow\RR$ with
$R_t(x)=R(x,t)$.

By construction, $\phi_1=\phi_h$ and $\phi_0|_{W_k}=\phi|_{W_k}$. Therefore
$\iota=f_1|_{W_k}$ is a symplectic embedding of $W_k$ into $(W,d\theta_h)$ such
that $\iota^*\theta_h=\theta|_{W_k}+dR_1$ and there are no critical points of
$\phi_h$ in the complement $W\setminus\iota(W_k)$.

We want to extend $\iota$ to a symplectomorphism
$\tilde{\iota}:(\hat{W}_k,d\hat{\theta}\rightarrow (W,d\theta_h)$. Begin by
replacing $\hat{\theta}$ by $\hat{\theta}_R=\hat{\theta}+dR_1$. Of course
$d\hat{\theta}=d\hat{\theta}_R$ still as $dR_1$ is closed. Let $Z_R$ be the
$d\hat{\theta}$-dual Liouville field to $\hat{\theta}_R$. Then, on $\iota(W_k)$
$\iota_*Z_R=Z_h$. If $\Phi^t_R$ and $\Phi^t_h$ denote the time $t$ Liouville
flows on $(\hat{W}_k,\hat{\theta}_R)$ and $(W,\theta_h)$ respectively then
\[\tilde{\iota}=\Phi^{t}_h\circ\iota\circ\Phi^{-t}_R:\hat{W}_k\rightarrow W\]
\noindent defines the required symplectomorphism. Since we have used $\Phi^t_R$
and $\Phi^{-t}_R(x)=\Phi^{-t}_h(x)$ for $x\in W_k$ (by equality of Liouville
vector fields), $\tilde{\iota}|_{W_k}=\iota$.
\end{proof}

We note another lemma for convenience.

\begin{lma}\label{otherplush}
If $\phi_1$ and $\phi_2$ are complete plurisubharmonic functions on $(W,J)$
with finitely many critical points then $(W,-dd^c\phi_1)$ is symplectomorphic
to $(W,-dd^c\phi_2)$.
\end{lma}
\begin{proof}
This follows from applying Lemma 5 from \cite{SS05} to the family
$t\phi_1+(1-t)\phi_2$: the associated 2-forms are all compatible with $J$ and
hence all symplectic.
\end{proof}

\subsection{Symplectomorphism groups}

\begin{dfn}
Let $(W,\omega)$ be a non-compact symplectic manifold and let $\mK$ be the set
of compact subsets of $W$. For each $K\in\mK$ let $\Symp_K(W)$ denote the group
of symplectomorphisms of $W$ supported in $K$, with the topology of
$\mC^{\infty}$-convergence. The group $\Symp_c(W,\omega)$ of
compactly-supported symplectomorphisms of $(W,\omega)$ is topologised as the
direct limit of these groups under inclusions.
\end{dfn}

One important fact we will tacitly and repeatedly use in the sequel is:
\begin{lma}
Let $C$ be compact and $f:C\rightarrow\Symp_c(W,\omega)$ be continuous. Then
the image of $f$ is contained in $\Symp_K(W,\omega)$ for some $K$.
\end{lma}

Let us restrict attention to the class of Liouville manifolds $(W,\theta,\phi)$
which are of finite-type and satisfy:
\begin{cnd}\label{wellgraded}
There is a $T\in[0,\infty]$, a $K\in[0,\infty)$ and an exhausting function
$f:[0,T)\rightarrow[K,\infty)$ such that:
\begin{itemize}
\item the Liouville flow of any point on $\phi^{-1}(K)$ is defined until time
$T$,
\item for all $t\in[0,T)$, the Liouville flow defines a diffeomorphism
$\phi^{-1}(K)\rightarrow\phi^{-1}(f(t))$.
\end{itemize}
\end{cnd}
Examples of such Liouville manifolds include Stein manifolds (where the
Liouville flow is the gradient flow of $\phi$ with respect to the K\"{a}hler
metric) and symplectic completions of Liouville domains.

A Liouville manifold $(W,\theta,\phi)$ comes with a canonical family of compact
subsets $W_r=\phi^{-1}[0,r]$ with natural inclusion maps
$\iota_{r,R}:W_r\hookrightarrow W_R$ for all $r<R$. This family is
\emph{exhausting}, that is $\bigcup_{r\in[0,\infty)}W_r=W$. When
$(W,\theta,\phi)$ satisfies condition \ref{wellgraded} we can use this family
to better understand the homotopy type of the symplectomorphism group
$\Symp_c(W,d\theta)$.

\begin{prp}\label{WGimpliesgood}
If $(W,\theta,\phi)$ is a Liouville manifold satisfying condition
\ref{wellgraded} and $K\in[0,\infty)$ is such that $W_K=\phi^{-1}[0,K)$
contains all critical points of $\phi$ then $\Symp_c(W_K,d\theta|_{W_K})$ is
weakly homotopy equivalent to $\Symp_c(W,d\theta)$.
\end{prp}
\begin{proof}
Let $\Phi^t$ denote the time $t$ Liouville flow of $Z$ where that is defined.
Let $L_{R-r}:\Symp_{W_R}(W,d\theta)\rightarrow\Symp_{W_r}(W,d\theta)$ be the
map:
\[L_{R-r}(\psi)=\Phi^{r-R}\circ\psi\circ\Phi^{R-r}\]
The inclusions $\iota_{r,R}$ induce inclusions
\[\iota^S_{r,R}:\Symp_{W_r}(W,d\theta)\hookrightarrow\Symp_{W_R}(W,d\theta)\]
for $r,R\geq K$. The maps $L_{R-r}$ are homotopy inverses for $\iota^S_{r,R}$.
Similarly, the inclusion
\[\iota^S_c:\Symp_c(W_K,d\theta|_{W_K})\rightarrow\Symp_{W_K}(W,d\theta)\]
has $L_{\epsilon}$ (for any small $\epsilon$) as a homotopy inverse.

We wish to prove that the inclusion
\[\iota:\Symp_c(W_K,d\theta|_{W_K})\rightarrow\Symp_c(W,d\theta)\]
is a weak homotopy equivalence. To see that the maps
\[\iota_{\star}:\pi_n(\Symp_c(W_K,d\theta|_{W_K}))\rightarrow\pi_n(\Symp_c(W,d\theta))\]
are surjective, it suffices to note that the image of any
$f:S^n\rightarrow\Symp_c(W,d\theta)$ lands in $\Symp_{W_R}(W,d\theta)$ for some
$R$ and that $\iota_{0,R}\circ\iota_c$ is a homotopy equivalence. To see
injectivity, any homotopy $h$ between maps
$f_1,f_2:S^n\rightarrow\Symp_c(W_K,d\theta|_{W_K})$ through maps
$S^n\rightarrow\Symp_c(W,d\theta)$ lands in some $\Symp_{W_R}(W,d\theta)$ and
$\iota_{0,R}\circ\iota_c$ is a homotopy equivalence so $h$ is homotopic to a
homotopy $f_1\simeq f_2$ in $\Symp_c(W_K,d\theta|_{W_K})$.
\end{proof}

\begin{prp}\label{twostein}
If $(W,J)$ is a complex manifold with two finite-type Stein structures $\phi_1$
and $\phi_2$ then $\Symp_c(W,-dd^c\phi_1)$ and $\Symp_c(W,-dd^c\phi_2)$ are
weakly homotopy equivalent.
\end{prp}
\begin{proof}
Let $\omega_i=-dd^c\phi_i$. By Proposition \ref{WGimpliesgood},
$\Symp_c(W,\omega_i)$ is weakly homotopy equivalent to
$\Symp_c(W_{K_i},\omega_i|_{W_{K_i}})$ for some $K_i$. Again by Proposition
\ref{WGimpliesgood}, this is weakly homotopy equivalent to
$\Symp_c(\hat{W}_{K_i},\hat{\omega}_i)$. Now by Lemma \ref{affinecomplete},
this is isomorphic to $\Symp_c(W,-dd^c\left(\phi_i\right)_h)$. But by Lemma
\ref{otherplush}, $(W,-dd^c\left(\phi_1\right)_h)$ is symplectomorphic to
$(W,-dd^c\left(\phi_2\right)_h)$ (since $\left(\phi_i\right)_h$ is complete).
This proves the proposition.
\end{proof}

\section{$\Symp_c(T^*\RP{2})$}\label{rp2}
\setcounter{unthm}{4}
\begin{thmun}
Let $\omega$ be the canonical symplectic form on the cotangent bundle of
$\RP{2}$. The group of compactly supported symplectomorphisms
$\Symp_c(T^*\RP{2})$ is weakly homotopy equivalent to $\ZZ$.
\end{thmun}

\begin{proof}
The proof is an adaptation of Seidel's proof \cite{Sei98} for $T^*S^2$. The 2-1
cover $S^2\rightarrow\RP{2}$ differentiates to a 2-1 cover of tangent bundles.
Identifying tangent and cotangent bundles via the round metric, we get a 2-1
cover $T^*S^2\rightarrow T^*\RP{2}$. Restrict this to a 2-1 cover of the unit
disc subbundles $U^*(S^2)\subset T^*S^2$ and $U^*(\RP{2})\subset T^*\RP{2}$:
\[\pi:U^*(S^2)\rightarrow U^*(\RP{2})\]
$\pi$ intertwines the cogeodesic flows on these cotangent disc bundles. Since
the cogeodesic flows are periodic, one can symplectically cut along them and
$\pi$ extends to a branched cover of the compactifications:
\[S^2\times S^2\rightarrow \PP{2}\]
branched along the diagonal $\Delta\subset S^2\times S^2$ (which maps
one-to-one onto a conic $C$ in $\PP{2}$). Here $\Delta$ and $C$ are the reduced
loci of the symplectic cut.

Let $\iota:S^2\times S^2\rightarrow S^2\times S^2$ be the involution which
swaps the two $S^2$ factors. This is the deck transformation of the branched
cover. Let:
\begin{itemize}
\item $\mS_0^{\iota}$ denote the group of $\iota$-equivariant
symplectomorphisms of $S^2\times S^2$ which fix $\Delta$ pointwise and act
trivially on the homology of $S^2\times S^2$,
\item $\mS^{\iota}_1\subset\mS_0^{\iota}$ denote the subgroup of
symplectomorphisms compactly supported in the complement of $\Delta$,
\item $\mS$ denote the group of symplectomorphisms of $\PP{2}$ compactly
supported in the complement of $C$. This is homotopy equivalent to the group we
are interested in.
\end{itemize}

An element $\tilde{\phi}\in\mS^{\iota}_1$ descends to an element $\phi\in\mS$.
$\tilde{\phi}$ is the only $\iota$-equivariant symplectomorphism of $S^2\times S^2$ which acts trivially on homology and descends to $\phi$, so the correspondence
$\tilde{\phi}\rightarrow\phi$ is a homeomorphism $\mS^{\iota}_1\rightarrow\mS$. We have therefore reduced the
proposition to computing the weak homotopy type of $\mS^{\iota}_1$.

Let $\mG$ be the group of gauge transformations of the symplectic normal bundle
to $\Delta\subset S^2\times S^2$. By a standard argument \cite{Sei08} this is
homotopy equivalent to $S^1$. Let $\mS^{\iota}_0\rightarrow\mG$ be the map
taking an $\iota$-equivariant symplectomorphism fixing $\Delta$ to its
derivative along the normal bundle of $\Delta$. This is a fibration whose fibre
is weakly homotopy equivalent to $\mS^{\iota}_1$. The proposition will follow
from the long exact sequence of this fibration if we can show that
$\mS^{\iota}_0$ is contractible. This is where Gromov's theorem comes in.

Let $J_0$ denote the product complex structure on $S^2\times S^2$. Let
$\mJ^{\iota}_0$ denote the space of $\omega$-compatible almost complex
structures $J$ on $S^2\times S^2$ such that:
\begin{itemize}
\item $J\circ\iota_*=\iota_*\circ J$,
\item $J$ restricted to $T\Delta$ is equal to $J_0$.
\end{itemize}
\begin{lma}\label{jcontrlma}
$\mJ^{\iota}_0$ is contractible.
\end{lma}
\begin{proof}
Recall that if $g$ is a metric on a manifold $X$ and $\omega$ a symplectic form
then the endomorphism $A\in\Endo(TX)$ defined by
\[\omega(\cdot,A\cdot)=g(\cdot,\cdot)\]
satisfies $A^{\dagger}=-A$, so $A^{\dagger}A$ is symmetric and positive
definite. Let $PA^{\dagger}P^{-1}=\mbox{diag}(\lambda_1,\ldots,\lambda_n)$ be
the diagonalisation of $A^{\dagger}A$. Define the square root
$\sqrt{A^{\dagger}A}$ to be the matrix
$P^{-1}\mbox{diag}(\sqrt{\lambda_1},\ldots,\sqrt{\lambda_n})P$. Then
$J_A=\sqrt{A^{\dagger}A}^{-1}A$ is an almost complex structure on $X$
compatible with $\omega$.

Let $J$ be an almost complex structure on $S^2\times S^2$ compatible with the
product form $\omega$. This defines a metric $g_0$ via
$g_0(\cdot,\cdot)=\omega(\cdot,J\cdot)$. Similarly, $\iota^*J$ is an
$\omega$-compatible almost complex structure which defines a metric $g_1$. The
space of metrics is convex and the metric $g_t=(1-t)g_0+tg_1$ defines a family
of endomorphisms $A_t$ as above interpolating between $J$ and $\iota^*J$. Now
$J_{A_t}$ is a family of almost complex structures interpolating between $J$
and $\iota^*J$ and $J_{A_{1/2}}$ is $\iota_*$-invariant. Hence the contractible
space of all $\omega$-compatible almost complex structures (equal to $J_0$
along $\Delta$) deformation retracts onto the space of $\iota_*$-equivariant
ones.
\end{proof}

We now prove that $\mS^{\iota}_0$ is contractible. There is a map $A:\mS^{\iota}_0\rightarrow\mJ^{\iota}_0$ which sends $\phi$ to
$\phi_*J_0$. The following argument constructs a left inverse $B$ for $A$.

Gromov's theorem gives $J$-holomorphic foliations $\mF_1^J$ and $\mF_2^J$ for
each $J\in\mJ^{\iota}_0$. In fact the two foliations must be conjugate under
$\iota$ since $\iota\mF_1^J$ is certainly a foliation by $J$-holomorphic
spheres in the homology class $\mF_2^J$ and postivity of intersections implies
it is the unique one.  We also know that, for any $J\in\mJ^{\iota}_0$, $\Delta$
is $J$-holomorphic so each leaf of $\mF_i^J$ intersects $\Delta$ in a single
point. Define a (manifestly $\iota$-equivariant) self-diffeomorphism $\psi_J$
of $S^2\times S^2$ by sending a point $p$ to the pair
$(\Lambda_1(p)\cap\Delta,\Lambda_2(p)\cap\Delta)$ where $\Lambda_i(p)$ denotes
the unique leaf of $\mF_i^J$ passing through $p$. $\psi_J$ conjugates the pair
$(\mF_1^J,\mF_2^J)$ with the standard pair $(\mF_1^{J_0},\mF_2^{J_0})$. Since
these standard foliations are $\omega$-orthogonal, it is not hard to see that
$\psi_J^*\omega$ is $J$-tame. Define $\omega_t=t\psi_J^*\omega+(1-t)\omega$, so
that
\[\dot{\omega}_t=\psi_J^*\omega-\omega\]
Since $\psi_J$ acts trivially on homology this expression is exact and there
exists $\sigma$ such that $d\sigma=\psi_J^*\omega-\omega$. If $X_t$ is
$\omega_t$-dual to $\sigma$ then the time $t$ flow of $X_t$ defines a sequence
of diffeomorphisms $\psi_t$ such that $\psi_0=\id$ and
$\psi_t^*\omega_t=\omega$. Therefore $\psi_{J,t}=\psi_J\circ\psi_t$ is an
isotopy from $\psi_J$ to a symplectomorphism $\psi_{J,1}$ of $\omega$. We want
$\psi_{J,t}$ to have the following properties:
\begin{itemize}
\item $\psi_{J,t}$ is $\iota$-equivariant for all $t$,
\item $\psi_{J,t}$ fixes $\Delta$ for all $t$,
\item if $J=\phi_*J_0$ then $\psi_{J,1}=\phi$.
\end{itemize}
These can be achieved by suitable choice of $\sigma$. The first two follow if
we require $\sigma$ to be $\iota^*$-invariant and to vanish when restricted to
$TX|_{\Delta}$. These properties will be achieved momentarily (in a manner such
that $\sigma$ varies smoothly with $J$). The third property follows if we can
take $\sigma\equiv 0$ whenever $d\sigma\equiv 0$; for then if $J=\phi_*J_0$,
the foliations $\mF_i^J$ are just $\phi(\mF_i^{J_0})$ so $\psi_J=\phi$ and
$\dot{\omega}_t=0=d\sigma$. Taking $\sigma=0$ yields
$\psi_{J,1}=\psi_{J,0}=\phi$.

Fix a metric $g$ and let $\sigma'$ be the unique $g$-coexact 1-form for which
$d\sigma'=\dot{\omega}_t$ (which exists and varies smoothly with $J$ by Hodge
theory). In order to make sure $\psi_{J,t}$ fixes $\Delta$, we want $\sigma$ to
vanish on the bundle $TX|_{\Delta}$. Since $\psi_J|_{\Delta}=\id$, $d\sigma'=0$
along $\Delta$. In particular we may choose a function $f$ on $\Delta$
with $df=\sigma'|_{\Delta}$ (as $H^1(\Delta;\RR)=0$) and this choice is unique if we require $\int_{\Delta} f\omega=0$. Fix an $\epsilon$ such
that the map $\exp_g:U_{\epsilon}\rightarrow X$ is injective on the radius-$\epsilon$ subbundle $U_{\epsilon}$ of the normal bundle
$T\Delta^{\bot}$.
Define
\[\tilde{f}(\exp(q,v))=f(\exp(q,0))+\sigma'_{\exp(q,0)}(\exp_*v).\]
This satisfies $d\tilde{f}=\sigma'$ when restricted to $TX|_{\Delta}$. Fix a
cut-off function $\eta$ which equals $1$ inside the radius-$\epsilon/2$
subbundle $U_{\epsilon/2}$ and equals zero near the boundary of $U_{\epsilon}$.
The function $F=\eta\tilde{f}$ is now globally defined and still satisfies
$dF=\sigma'$ when restricted to $TX|_{\Delta}$. The 1-form
\[\sigma:=\frac{(\sigma'-dF)+\iota^*(\sigma'-dF)}{2}\]
is now $\iota^*$-invariant, vanishes along $\Delta$ and satisfies
$d\sigma=\dot{\omega}_t$.

We finally note that $\sigma$ varies smoothly with $J$ and that the unique
coexact $\sigma$ with $d\sigma=0$ is $\sigma=0$, so the third property holds (by Hodge theory).

This proves that $\mS^{\iota}_0$ is homotopy equivalent to $\mJ^{\iota}_0$ and hence contractible by Lemma \ref{jcontrlma}. The proposition now follows from the homotopy long exact sequence of the fibration $\mS^{\iota}_0\rightarrow\mG$, whose fibre is weakly equivalent to $\mS^{\iota}_1$ by the symplectic neighbourhood theorem.
\end{proof}

\begin{rmk} Since all geodesics of $\RP{2}$ are closed one can define a Dehn
twist in a Lagrangian $\RP{2}$ just as one does for $S^2$. This is actually the
generator of the symplectic mapping class group $\ZZ$, see \cite{Sei98}.
\end{rmk}

\section{Groups associated to configurations of spheres}

This section reviews the topology of some well-known (Fr\'{e}chet-) Lie groups which will crop up frequently later.

\subsection{Gauge transformations}\label{gau}

Let $C\subset X$ be an embedded symplectic 2-sphere in a symplectic 4-manifold and fix a set $\{x_1,\ldots,x_k\}$ of distinct points of $C$. The normal bundle $\nu=TC^{\omega\bot}$ is a $SL(2,\RR)$-bundle over $C$. We will be interested in the group $\mG_k$ of symplectic gauge transformations of $\nu$ which equal the identity at $x_1,\ldots,x_k$. This will arise when we consider symplectomorphisms of $X$ which fix $C$ pointwise and are required to equal the identity at the $k$ specified points on $C$.

We first observe that $SL(2,\RR)$ deformation retracts to the subgroup $U(1)$ which is homeomorphic to a circle. The map
\[\ev_{x_1}:\mG_0\rightarrow SL(2,\RR)\]
taking a gauge transformation to its value at the point $x_1$ is a fibration whose fibre is $\mG_1$. By definition, $\mG_1$ is the space of based maps $S^2\rightarrow SL(2,\RR)$ so
\[\pi_i(\mG_1)=\pi_0(\Map(S^{2+i},SL(2,\RR)))=\pi_0(\Map(S^{2+i},S^1))=H^1(S^{2+i};ZZ)=0\]
and $\mG_1$ is weakly contractible. By the long exact sequence of the fibration $\ev_{x_1}$ we see that $\ev_{x_1}$ is a weak homotopy equivalence.

Take the map $\ev_{x_1}\times\cdot\times\ev_{x_i}:\mG_1\rightarrow SL(2,\RR)^i$. This is again a fibration whose fibre is $\mG_i$ and whose total space is contractible. The homotopy long exact sequence of this fibration implies that $\pi_0(\mG_i)=\ZZ^{i-1}$ for $i\geq 1$ and $\pi_j(\mG_i)=0$ for $j>0$.

We can identify generators for $\pi_0(\mG_k)$ as follows. For each point $x_i$ the map $\ev_{x_i}:\mG_{k-1}\rightarrow SL(2,\RR)$ yields a connecting homomorphism $\ZZ=\pi_1(SL(2,\RR))\rightarrow\pi_0(\mG_k)$. The symplectic form gives an orientation of $SL(2,\RR)$ so there is a preferred element $1\in\ZZ$. Let $g_C(x_i)\in\mG_k$ be the image of this element in $\pi_0(\mG_k)$. They are not independent, but they are canonical and they generate.

To summarise:
\begin{align*}
\mG_0 &\simeq S^1 \\
\mG_1 &\simeq \star \\
\mG_k &\simeq \ZZ^{k-1},\ k>1
\end{align*}

\subsection{Surface symplectomorphisms}\label{sympC}

Once again, let $C\subset X$ be an embedded symplectic 2-sphere and $\{x_1,\ldots,x_k\}$ a set of $k$ distinct points. Let $\Symp(C,\{x_i\})$ denote the group of symplectomorphisms of $C$ fixing the points $x_i$. Since $\Symp(C)$ acts $N$-transitively on points for any $N$, we may as well write $\Symp(C,k)$. Moser's theorem tells us that $\Symp(C,k)$ is a deformation retract of the group $\mD_k$ of diffeomorphisms of $S^2$ fixing $k$ points. Using techniques of Earle and Eells \cite{EE67} coming from Teichm\"{u}ller theory, one can prove:
\begin{itemize}
\item $\mD_1\simeq S^1$,
\item $\mD_2\simeq S^1$,
\item $\mD_3\simeq \star$,
\item $\mD_5\simeq P\Br(S^2,5)/\ZZ/2$, where $P\Br(S^2,5)$ is the pure braid group for 5-strands on $S^2$ and the $\ZZ/2$ is generated by a full twist (see section \ref{delpres}).
\end{itemize}
Here the $S^1$ in $\mD_1$ can be thought of as rotating around the fixed point.
$\mD_2$ has a map to $S^1\times S^1$ measuring the angle of rotation of a
diffeomorphism about the two fixed points $p_1$ and $p_2$. The kernel of this
map is the group $\mD_2^c$ of compactly supported diffeomorphisms of $S^2\setminus\{p_1,p_2\}$, which is
weakly equivalent to $\ZZ$ generated by a Dehn twist in a simple closed curve
generating the fundamental group of $S^2\setminus\{p_1,p_2\}$. In the homotopy
long exact sequence of the fibration $\mD_2^c\rightarrow\mD_2\rightarrow
(S^1)^2$, one gets:
\[1\rightarrow\pi_1(\mD_2)\rightarrow\ZZ^2\rightarrow\ZZ\rightarrow\pi_0(\mD_2)\rightarrow
1\]
The Dehn twist is the image of $(1,0)$ and $(0,1)$ under the map
$\ZZ^2\rightarrow\ZZ$, so $\mD_2\simeq S^1$. If $p_1$ and $p_2$ are antipodal, a $2\pi$ rotation around the axis through them gives a non-trivial loop.

\subsection{Configurations of spheres}\label{confC}

We will frequently meet the situation where there is a configuration of embedded symplectic 2-spheres $C=C_1\cup\cdots\cup C_n\subset X$ in a 4-manifold. Write $I$ for the set of intersection points amongst the components. Suppose that there are no triple intersections amongst components and that all intersections are transverse. Let
\begin{itemize}
\item $\Stab(C)$ denote the group of symplectomorphisms of $X$ fixing the components of $C$ pointwise.
\item $k_i$ denote the number of intersection points in $I\cap C_i$ and $\Symp(C_i,k_i)$ denote the group of symplectomorphisms of components of $C$ fixing all the intersection points. Write $\Symp(C)$ for the product $\prod_{i=1}^n\Symp(C_i,k_i)$.
\item $\mG(C_i)$ denote the group of gauge transformations of the normal bundle to $C_i\subset X$ which equal the identity at the $k_i$ intersection points (so $\mG(C_i)\cong\mG_{k_i}$). Write $\mG$ for the product $\prod_{i=1}^n\mG(C_i)$.
\end{itemize}

Let us write $\Stab^0(C)$ for the kernel of the map
\begin{align*}
\Stab(C)&\rightarrow\Symp(C) \\
\phi&\mapsto \left(\phi|_{C_1},\ldots,\phi|_{C_n}\right)
\end{align*}
We will need to understand the homotopy long exact sequence of the fibration
\begin{equation}\label{stab0fib}\tag{*}
\begin{CD}
\Stab^0(C) @>>> \Stab(C) \\
@. @VVV \\
@. \Symp(C)
\end{CD}
\end{equation}
namely
\[\cdots\rightarrow\pi_1(\Symp(C))\rightarrow\pi_0(\Stab^0(C))\rightarrow\pi_0(\Stab(C))\rightarrow\cdots\]

There is a map $\Stab^0(C)\rightarrow\mG$ which sends a symplectomorphism $\phi$ to the induced map on the normal bundles of components of $C$. We can understand the composite
\[\pi_1(\Symp(C))\rightarrow\pi_0(\Stab^0(C))\rightarrow\pi_0(\mG)\]
by thinking purely locally in a neighbourhood of $C$.

To see this, note that $\pi_1(\Symp(C))$ is only non-trivial if $C$ contains components with $n_i=1$ or $2$ marked points. Suppose $n_i=1$. There is a Hamiltonian circle action which rotates $C_i$ around the marked point, giving a generating loop $\gamma$ in $\pi_1(\Symp(C_i,n_i))$. Pull this back to a Hamiltonian circle action on the normal bundle of $C_i$ in $X$. By the symplectic neighbourhood theorem this is a local model for $X$ near $C_i$, so by cutting off the Hamiltonian at some radius in the normal bundle one obtains a path $\phi_t$ in $\Symp(X)$ consisting of symplectomorphisms which are supported in a neighbourhood of $C_i$ (see figure). The symplectomorphism $\phi_t$ can be assumed to fix $C$ as a set (indeed it just rotates the component $C_i$) and $\phi_{2\pi}$ fixes $C$ pointwise, i.e. $\phi_{2\pi}\in\Stab^0(C)$. By definition, $\phi_{2\pi}$ represents the image of $\gamma\in\pi_1(\Symp(C))$ under the boundary map $\pi_1(\Symp(C))\rightarrow\pi_0(\Stab^0(C))$ of the long exact sequence above. A similar story holds when $n_i=2$. The following lemma is immediate from the definitions.

\begin{lma}\label{localcomp}
The image of $\phi_{2\pi}$ under the map $\pi_0(\Stab^0(C))\rightarrow\pi_0(\mG)$ is
\begin{itemize}
\item ($n_i=1$:) $g_{C_j}(x)\in\pi_0(\mG(C_j))$ where $x\in C_i\cap C_j$,
\item ($n_i=2$:) $g_{C_j}(x)g_{C_k}(y)\in\pi_0(\mG(C_j))\times\pi_0(\mG(C_k))$ where $x\in C_i\cap C_j$ and $y\in C_i\cap C_k$.
\end{itemize}
\end{lma}

\subsection{Banyaga's isotopy extension theorem}

We will make constant recourse to the following ($k$-parameter version of a) theorem of Banyaga:

\begin{thm}[Banyaga's isotopy extension theorem, see \cite{MS05}, p. 98]\label{banyaga}
Let $(X,\omega)$ be a compact symplectic manifold and $C_s\subset X$ be a $k$-parameter family of compact subsets ($s\in [0,1]^k$). Let $\phi_{t,s}:U\rightarrow X$ be a $k$-parameter family of symplectic isotopies of open sets such that $U_s:=\phi_{0,s}(U)$ is a neighbourhood of $C_s$ and assume that for all $s$
\[H^2(X,C_s;\RR)=0\]
Then there are neighbourhoods $\mathcal{N}_s\subset U_s$ of $C_s$ and a $k$-parameter family of symplectic isotopies $\psi_{t,s}:X\rightarrow X$ such that $\psi_{t,s}|_{\mathcal{N}_s}=\phi_{t,s}|_{\mathcal{N}_s}$ for all $t$ and $s$.
\end{thm}

The proof uses a standard Moser-type argument. This theorem will be used to prove the homotopy lifting property on a number of occasions.

\section{$\Symp_c(\CC^*\times\CC)$}\label{cstarcsect}

In this section we prove the following theorem:
\begin{thmun}\label{cstarcthm}
Let $\omega$ be the product symplectic form on $\CC^*\times\CC$. The group of
compactly supported symplectomorphisms $\Symp_c(\CC^*\times\CC)$ is weakly
contractible.\end{thmun}

\subsection{Proof of Theorem \ref{cstarcthm}}

We define three holomorphic spheres in $X=\PP{1}\times\PP{1}$:
\begin{align*}
C_1 &= \{0\}\times\PP{1} \\
C_2 &=\{\infty\}\times\PP{1} \\
C_3 &= \PP{1}\times\{\infty\}.
\end{align*}
Consider the complements
\begin{align*}
U&=X\setminus C_1\cup C_2\cup C_3 & U'&=X\setminus C_2\cup C_3.
\end{align*}
These are biholomorphic to $\CC^*\times\CC$ and $\CC^2$ respectively. The split symplectic form $\omega$ on $X$ restricts to symplectic forms $\omega|_U$ and $\omega|_{U'}$ which are both Stein for the standard (product) complex structures. By Proposition \ref{twostein},
\begin{align*}
\Symp_c(U)&\simeq\Symp_c(\CC^*\times\CC) & \Symp_c(U')&\simeq\Symp_c(\CC^2)
\end{align*}
and by a theorem of Gromov \cite{Gro85}, $\Symp_c(\CC^2)\simeq\star$.

\[
\xy
(0,0)*{};(20,0)*{} **\dir{-};
(0,0)*{};(0,-20)*{} **\dir{-};
(20,0)*{};(20,-20)*{} **\dir{-};
(0,-20)*{};(20,-20)*{} **\dir{-};
(-3,-10)*{C_1};(23,-10)*{C_2};(10,2)*{C_3};
(10,-10)*{X};
(30,0)*{};(50,0)*{} **\dir{--};
(30,0)*{};(30,-20)*{} **\dir{--};
(50,0)*{};(50,-20)*{} **\dir{--};
(30,-20)*{};(50,-20)*{} **\dir{-};
(40,-10)*{U};
(60,0)*{};(80,0)*{} **\dir{--};
(60,0)*{};(60,-20)*{} **\dir{-};
(80,0)*{};(80,-20)*{} **\dir{--};
(60,-20)*{};(80,-20)*{} **\dir{-};
(70,-10)*{U'};
\endxy
\]

We will now define a space on which $\Symp_c(U')$ acts. Let $\mJ$ denote the space of $\omega$-compatible almost complex structures on $X$.

\begin{dfn}
A \emph{standard configuration} in $X$ is an embedded symplectic sphere $S$ satisfying the following properties:
\begin{itemize}
\item $S$ is homologous to $C_1$,
\item $S$ is disjoint from $C_2$,
\item there exists a $J\in\mJ$ making $S$, $C_2$ and $C_3$ simultaneously $J$-holomorphic,
\item there exists a neighbourhood $\nu$ of $C_3$ such that $\nu\cap S$ is equal to $\nu\cap C_1$.
\end{itemize}
Let $\mC_0$ denote the space of standard configurations, topologised as a subset of the quotient
\[\Map(\PP{1},X)/\Diff(\PP{1})\]
where $\Map(\PP{1},X)$ and $\Diff(\PP{1})$ are given the $\mC^{\infty}$-topology.
\end{dfn}

The first important fact about this space is the following proposition whose proof is postponed to section \ref{prfcstarccontr}:

\begin{prp}\label{cstarccontr}
$\mC_0$ is weakly contractible.
\end{prp}

Given this proposition, we observe that

\begin{lma}
$\Symp_c(U')$ acts transitively on $\mC_0$.
\end{lma}
\begin{proof}
Since $\mC_0$ is weakly contractible it is path-connected. If $S_0$ and $S_1$ are two standard configurations let $S_t$ be a path connecting them. Then
\[T_t=S_t\cup C_2\cup C_3\]
is an isotopy of configurations of symplectic spheres. Since $S_t$ is a standard configuration for each $t$ the isotopy $T_t$ extends to an isotopy of a neighbourhood of $T_t$ by the symplectic neighbourhood theorem. We may assume this isotopy fixes a neighbourhood of $C_2\cup C_3$ pointwise. Since $H^2(X,T_t;\RR)=0$, Banyaga's Theorem \ref{banyaga} above implies that this isotopy extends to a path $\psi_t$ in $\Symp(X)$. By construction $\psi_t$ is the identity on a neighbourhood of $C_2\cup C_3$. Hence $\psi_1\in\Symp_c(U')$ sends $S_0$ to $S_1$, proving transitivity.
\end{proof}

Morevover, the map $\Symp_c(U')\rightarrow\mC_0=\Orb(C_1)$ is a fibration by Banyaga's theorem. We identify the stabiliser:

\begin{lma}\label{stabwkcontr}
The stabiliser $\Stab(C_1)$ of $C_1$ under this action is weakly homotopy equivalent to $\Symp_c(U)$.
\end{lma}
\begin{proof}
Since a symplectomorphism $\phi\in\Symp_c(U')$ fixes a neighbourhood of $C_3$, an element $\phi\in\Stab(C_1)\subset\Symp_c(U')$ will induce a symplectomorphism $\bar{\phi}$ of $C_1$ which is compactly supported away from the point $\infty=([0:1],[1:0])\in C_1\cap C_3$. The map $\phi\mapsto\bar{\phi}$ gives a fibration
\[
\begin{CD}
\Stab^0(C_1) @>>> \Stab(C_1) \\
@. @VVV \\
@. \Symp_c(C_1,\infty)
\end{CD}
\]
where $\Stab^0(C_1)$ is the group of symplectomorphisms $\phi\in\Symp_c(U')$
which fix $C_1$ pointwise and $\Symp_c(C_1,\infty)$ denotes the group of
compactly supported symplectomorphisms of $C_1\setminus\{\infty\}$. This latter group is contractible, so it suffices to prove that $\Stab^0(C_1)$ is weakly homotopy equivalent to $\Symp_c(U)$.

Define $\mG$ to be the group of symplectic gauge transformations of the normal bundle to $C_1$ in $X$ which equal the identity on some neighbourhood of $\infty\in C_1$. It follows from section \ref{gau} that this group is weakly contractible. The map $\Stab^0(C_1)\rightarrow\mG$, which takes a compactly supported symplectomorphism of $U'$ fixing $C_1$ to its derivative on the normal bundle to $C_1$, defines a fibration. Weak contractibility of $\mG$ implies that the kernel, $\Stab^1(C_1)$, of this fibration is weakly homotopy equivalent to $\Stab^0(C_1)$.

This kernel is the group of compactly supported symplectomorphisms of $U'$
which fix $C_1$ and act trivially on its normal bundle. By the symplectic
neighbourhood theorem, this is weakly homotopy equivalent to the group of
compactly supported symplectomorphisms of $U=U'\setminus C_1$.
\end{proof}

Putting all this together, weak contractibility of $\Symp_c(U)$ follows from the long exact sequence of the fibration
\begin{equation}\label{orbstabCtimesCstar}\tag{**}
\begin{CD}
\Symp_c(U)\simeq\Stab(C_1) @>>> \Symp_c(U')\simeq\star \\
@. @VVV \\
@. \mC_0\simeq\star
\end{CD}
\end{equation}

We noted earlier that $\Symp_c(U)\simeq\Symp_c(\CC^*\times\CC)$, so this proves Theorem \ref{cstarcthm}.

\subsection{Proof of Proposition \ref{cstarccontr}}\label{prfcstarccontr}

To prove weak contractibility of the space of standard configurations, we will need some preliminary work.

\subsubsection{Gompf isotopy}

We first introduce another space of configurations of spheres.

\begin{dfn}
A \emph{nonstandard configuration} in $X$ is an embedded symplectic sphere $S$ satisfying the following properties:
\begin{itemize}
\item $S$ is homologous to $C_1$,
\item $S$ is disjoint from $C_2$,
\item $S$ intersects $C_3$ transversely once at $([0:1],[1:0])$,
\item there exists a $J\in\mJ$ making $S$, $C_2$ and $C_3$ simultaneously $J$-holomorphic,
\end{itemize}
Let $\mC$ denote the space of nonstandard configurations, topologised as a quotient
\[\Map(\PP{1},X)/\Diff(\PP{1})\]
where $\Map(\PP{1},X)$ and $\Diff(\PP{1})$ are given the $\mC^{\infty}$-topology.
\end{dfn}

These are just like standard configurations but without the requirement that they intersect $C_3$ in a standard way. An additional subtlety arises from requiring there exists a $J$ making $S$ and $C_3$ simultaneously holomorphic: it is not even true that two transverse symplectic linear subspaces of a symplectic vector space can be made simultaneously $J$-holomorphic. This is discussed further below.

 There is an inclusion $\iota:\mC_0\hookrightarrow\mC$.
\begin{lma}\label{gompfy}
$\iota$ is a weak homotopy equivalence.
\end{lma}
The proof of this uses a construction due to Gompf. Specifically, the
proof of Lemma 2.3 from \cite{Gom95} implies the following:

\begin{lma}\label{gompfiso}
Let $N$ be a symplectic 2-manifold and $E\rightarrow N$ be a (disc-subbundle of a) rank-2 vector bundle with a symplectic form $\omega$ on the total space making the zero-section both
symplectic and symplectically orthogonal to the fibres. Let $M$ be a 2-dimensional
symplectic submanifold of $E$, closed as a subset of $E$, intersecting $N$
transversely at a single point $x$. Then there is a symplectic isotopy $M(t)$ of $M$
for which $M(0)=M$, $M(t)=M$ outside a small neighbourhood of $M(t)\cap N$ and
$M(1)$ agrees with the fibre $E_x$ in a small neighbourhood of $x$.
\end{lma}

The isotopy defined in that proof depends on: a choice of neighbourhood of $N$
in which $M$ intersects $N$ once transversely and three auxiliary real parameters. Other choices (such as a cut-off function $\mu$) can be
made to depend only on these real parameters. It is clear from the definitions
of these parameters that they can be chosen to depend continuously on $M$ as
$M$ varies in the space of symplectic submanifolds with the
$\mC^{\infty}$-topology, and the same neighbourhood of $N$ can be used for
$M_1$ and $M_2$ which are close in the $\mC^{\infty}$-topology. We will call the
isotopy defined in that proof a \emph{Gompf isotopy}.

It also follows from Gompf's construction that Gompf isotopy preserves the space $\mC$. The only subtlety is that $\mC$ consists of configurations whose components can be made simultaneously $J$-holomorphic. One must check that at the intersection points, the tangent planes of the intersecting components can be made simultaneously $J$-holomorphic. In a Darboux chart ($\cong\RR^2\times\RR^2$) around the intersection point, where one of the intersecting components is mapped to the $\RR^2\times\{\star\}$ coordinate plane, transversality of the intersection implies one can write the other component as a graph of a map $f:\RR^2\rightarrow\RR^2$. The condition that this graph is symplectic is just that $\det(f_{\star})>-1$. The condition that the graph can be made simultaneously $J$-holomorphic with $\RR^2\times\{\star\}$ is that $\det(f_{\star})\geq 0$. In the local model used to define it, Gompf isotopy changes $\det(f_{\star})$ monotonically towards 0, so it preserves $\mC$.

\begin{proof}[Proof of Lemma \ref{gompfy}]
Let $f_1$, $f_2:(S^n,\star)\rightarrow(\mC_0,C)$ represent homotopy classes
$[f_1]$, $[f_2]$ in $\pi_n(\mC_0,C)$. We first show that
$\iota_{\star}[f_1]=\iota_{\star}[f_2]$ only if $[f_1]=[f_2]$. Suppose
$S:(S^n\times [1,2],\star)\rightarrow(\mC,C)$ is a homotopy between $\iota\circ
f_1$ and $\iota\circ f_2$.

By the symplectic neighbourhood theorem there is a neighbourhood $\nu$ of $C_3$, isomorphic to a disc-subbundle of $C_3\times\CC$ and such that $S(x,t)\cap\nu$ satisfies the hypotheses of Lemma \ref{gompfiso}. Pick the parameters small enough to define a Gompf isotopy $\Gamma$ for all spheres $S(x,t)$. Then
\[\Gamma:\left(S^n\times[1,2]\times[1,2],\star\right)\rightarrow(\mC,C)\]
gives a homotopy from $S$ to $S'=\Gamma(\cdot,\cdot,2)$. This new homotopy lies entirely in $\mC_0$. Since the original maps $f_i:S^n\rightarrow\mC_0$ landed in the space of standard configurations and Gompf isotopy preserves the property of being a standard configuration, $S'$ is indeed a homotopy connecting $f_1$ and $f_2$ in $\mC_0$. This
proves injectivity of $\iota_{\star}$.

To prove surjectivity, let $f:(S^n,\star)\rightarrow(\mC,C)$ represent a
homotopy class $[f]$. For some neighbourhood $\nu$ of $C_3$ and sufficiently small choice of parameters one obtains a Gompf isotopy rel $C$ of the image of $f$ into
$\mC_0$. This gives a homotopy class $[f']$ in $\pi_n(\mC_0,C)$ and the Gompf
isotopy is a homotopy between $\iota_{\star}[f']$ and $[f]$.

Since $\iota_{\star}$ is an isomorphism on homotopy groups, $\iota$ is a weak
homotopy equivalence.
\end{proof}

\subsubsection{Structures making a configuration holomorphic}\label{confhol}

Let $S=\bigcup_{i=1}^nS_i$ be a union of embedded symplectic 2-spheres in a symplectic 4-manifold $X$. Suppose that the various components intersect transversely and that there are no triple intersections. Suppose further that there is an $\omega$-compatible $J$ for which all the components $S_i$ are $J$-holomorphic. Let $\mH(S)$ denote the space of $\omega$-compatible almost complex structures $J$ for which all components of $S$ are simultaneously $J$-holomorphic.

\begin{lma}\label{contrH}
If the components of $S$ intersect one another symplectically orthogonally then the space $\mH(S)$ is weakly contractible.
\end{lma}

This lemma is proved in appendix \ref{spaceofstructures}.

\subsubsection{Gromov's theory of pseudoholomorphic curves}

The following theorem follows from Gromov's theory of pseudoholomorphic curves in symplectic 4-manifolds.

\begin{thm}[\cite{Gro85}]
Given an $\omega$-compatible almost complex structure $J$ on $X=\PP{1}\times\PP{1}$ there is a unique $J$-holomorphic curve through $([0:1],[1:0])$ in the homology class $[C_1]$.
\end{thm}

If we restrict to the space $\mH(C')$ of almost complex structures which make $C_2$ and $C_3$ into $J$-holomorphic spheres then this gives us a map
\[G:\mH(C')\rightarrow\mC\]
since by positivity of intersections, the unique $J$-holomorphic curve through $([0:1],[1:0])$ in the homology class $C_1$ is then a nonstandard configuration, i.e. it cannot intersect $C_2$ and it must intersect $C_3$ transversely once.

\subsubsection{Proof of Proposition \ref{cstarccontr}}

Let $f:(S^n,\star)\rightarrow(\mC_0,C)$ be a based map. We must show that $f$ is nullhomotopic. Let $\infty\in S^n$ be the antipode to $\star$ and let $R$ be the space of great half-circles $r(t)$ connecting $r(0)=\star$ to $r(1)=\infty$. Each such half-circle defines an isotopy $f(r(t))$ of standard configurations. This isotopy extends to a small neighbourhood of the standard configurations by the symplectic neighbourhood theorem and then to a global isotopy $\psi_{r,t}\in\Symp(C)$ by Banyaga's isotopy extension theorem (since $H^2(X,S\cup C_3;\RR)=0$ for all the standard configurations). We can assume this isotopy fixes a neighbourhood of $C_2$ and $C_3$ pointwise.

Suppose $J_0$ is an $\omega$-compatible almost complex structure for which the basepoint $C\in\mC_0$ is $J_0$-holomorphic. Then the standard configuration $f(r(t))$ is $\left(\psi_{r,t}\right)_{\star}J_0$-holomorphic. Let $\overline{B}^n=S^n\setminus\{\infty\}\cup S^{n-1}$ be the compactification of $S^n\setminus\{\infty\}$ which adds an endpoint to every open half-circle $\{r(t)\}_{t\in[0,1]}$ (this is the oriented real blow-up of $S^n$ at $\infty$). Define a map: 
\begin{align*}
\tilde{f}_1:&\overline{B}^n\rightarrow \mH(C') \\
\tilde{f}_1(r,t)&=\left(\psi_{r,t}\right)_{\star}J_0\in\mH(C')
\end{align*}
which lifts $\overline{B}^n\rightarrow S^n\stackrel{f}{\rightarrow}\mC_0$ (where the first map collapses $S^{n-1}$ to the point $\infty$).
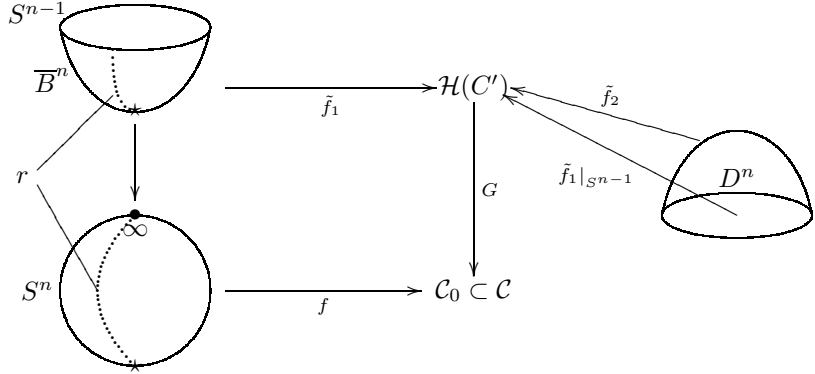
\[
\xy
(0,5)*\xycircle(10,3){-};
(-10,5)*{};(10,5)*{} **\crv{(-7,-10) & (7,-10)};
(0,-6)*{\star};(0,-20)*{\bullet};(0,-22)*{\infty};(0,-40)*{\star};
(0,-30)*\xycircle(10,10){-};
{\ar_{\tilde{f}_1}  (12,-3)*{}; (40,-3)*{}};
{\ar_{f} (12,-30)*{}; (38,-30)*{}};
{\ar (0,-8)*{}; (0,-18)*{}};
{\ar^{G} (45,-5)*{}; (45,-28)*{}};
(45,-30)*{\mC_0\subset\mC};(45,-3)*{\mH(C')};
(0,-40)*{};(0,-20)*{} **\crv{~*=<2pt>{.}(-10,-30)};
(0,-6)*{};(-3,2)*{} **\crv{~*=<2pt>{.}(-3,-6)};
(-15,-15)*{r};
(-13,-14)*{};(-3,-4)*{} **\dir{-};
(-13,-16)*{};(-5,-30)*{} **\dir{-};
(-13,7)*{S^{n-1}};(-13,-30)*{S^n};(-11,-2)*{\overline{B}^n};
(80,-20)*\xycircle(10,3){-};
(70,-20)*{};(90,-20)*{} **\crv{(73,-5) & (87,-5)};
{\ar_{\tilde{f}_2}  (75,-10)*{}; (50,-3)*{}};
{\ar^{\tilde{f}_1|_{S^{n-1}}} (80,-20)*{}; (49,-4)*{}};
(80,-15)*{D^n};
\endxy
\]
The image of the restriction $\tilde{f}_1|_{S^{n-1}}:S^{n-1}\rightarrow\mH(C')$ consists of $\omega$-compatible almost complex structures which make the standard configuration $f(\infty)$ holomorphic. The space of such almost complex structures ($\mH(f(\infty))\subset\mH(C')$) is contractible by Lemma \ref{contrH} and hence there is a map $\tilde{f}_2:D^n\rightarrow\mH(C')$ with $\tilde{f}_2|_{\partial D^n}=\tilde{f}|_{S^{n-1}}$ (where $D^n$ is the closed $n$-ball) such that $G\circ\tilde{f}_2(x)=f(\infty)$. Glue $\tilde{f}_1$ and $\tilde{f}_2$ to obtain a map $\tilde{f}=\tilde{f}_1\cup\tilde{f}_2:S^n\rightarrow \mH(C')$ for which $G\circ\tilde{f}$ is homotopic to $f$.

By Lemma \ref{contrH} $\mH(C')$ is contractible, $\tilde{f}$ is nullhomotopic via a map $H:D^{n+1}\rightarrow\mH(C')$. Therefore $G\circ H$ is a nullhomotopy of $f$ in $\mC$. But we have observed (Lemma \ref{gompfy}) that $\mC_0\rightarrow\mC$ is a weak homotopy equivalence, hence $f$ is nullhomotopic in $\mC_0$.

\section{Del Pezzo surfaces}\label{delp}

\subsection{Results}\label{delpres}

In this section, we use known results on symplectomorphism groups of affine
varieties to calculate the weak homotopy type of the symplectomorphism groups
of some Del Pezzo surfaces. We recall:

\begin{dfn}
A symplectic del Pezzo surface is the symplectic manifold underlying a smooth
complex projective surface with ample anticanonical line bundle. Equivalently
it is one of the following symplectic 4-manifolds:

\begin{itemize}
\item $Q=S^2\times S^2$ equipped with the product of the area-1 Fubini-Study
forms on each factor (this also arises as the K\"{a}hler form on a smooth
quadric hypersurface in $\PP{3}$),
\item A symplectic blow-up $\DD_n$ of $\PP{2}$ (with its anticanonical
symplectic form $3\omega_{FS}$) in $n<9$ symplectic balls of equal volume such
that
\[\omega_{\DD_n}(E_k)=\omega_{FS}(H)\]
\noindent where $E_k$ is an exceptional sphere and $H$ is a line in $\PP{2}$.
\end{itemize}
\end{dfn}

We prove the following.
\setcounter{unthm}{2}
\begin{thmun}
Let $\Symp_0(X)$ denote the group of symplectomorphisms of $X$ acting trivially
on homology and let $\Diff^+(S^2,5)$ denote the group of diffeomorphisms of
$S^2$ fixing five points.
\begin{itemize}
\item $\Symp(\DD_3)\simeq T^2$,
\item $\Symp(\DD_4)\simeq \star$,
\item $\Symp(\DD_5)\simeq \Diff^+(S^2,5)$.
\end{itemize}
\noindent where all $\simeq$ signs denote weak homotopy equivalence.
\end{thmun}

The third result in the theorem deserves some comment. Seidel has shown
\cite{Sei08} that there are maps $C/G\rightarrow B\Symp(\DD_5)\rightarrow C/G$
whose composition is homotopic to the identity. Here $C$ is the configuration
space of five ordered points on $S^2$ and $G$ is the group $\mathbb{P}SL_2\CC$
of M\"{o}bius transformations of $S^2$. $G$ acts freely on $C$, so we get a
principal $G$-bundle:
\[
\begin{CD}
G @>a>> C \\
@. @VVV \\
@. C/G
\end{CD}
\]
The group $\Diff^+(S^2)$ also acts on $C$ and there is an evaluation fibration
\[
\begin{CD}
\Diff^+(S^2,5) @>>> \Diff^+(S^2) \\
@. @VVbV \\
@. C
\end{CD}
\]
The maps $a$ and $b$ are actually homotopy equivalent (in view of the fact that
the inclusion $G\rightarrow\Diff^+(S^2)$ is a homotopy equivalence), so $C/G$
is weakly homotopy equivalent to $B\Diff^+(S^2,5)$. In fact, Teichm\"{u}ller theory
shows that $\Diff^+(S^2,5)$ is weakly homotopy equivalent to its group $B$ of
components. We can find the group $B$ as follows. $\pi_1(C)$ is the group
$P\Br(S^2,5)$ of five-strand pure braids on the sphere. $\pi_1(G)$ is $\ZZ/2$,
generated by a loop of rotations through $0$ to $2\pi$. By the homotopy long
exact sequence associated to the fibration $G\rightarrow C\rightarrow C/G$,
$\pi_1(C/G)$ is the cokernel of the map $\ZZ/2\rightarrow P\Br(S^2,5)$. It is
not hard to see that the image of this map is the full twist $\tau$ (which has
order 2 in the braid group), so $B\cong P\Br(S^2,5)/\langle\tau\rangle$.

\subsection{Outline of proof}

The proofs in all three cases run along similar lines. We outline the general
picture before filling in details individually. Therefore let $(X,\omega)$
denote any of $\DD_3$, $\DD_4$ or $\DD_5$, let $\mJ$ denote the space of
$\omega$-tame almost complex structures on $X$ and let $\Symp_0(X)$ denote the
group of symplectomorphisms acting trivially on $H^*(X,\ZZ)$.

In each case we will identify a divisor $C=\bigcup_i C_i\subset X$ such
that
\begin{itemize}
\item $C$ consists of embedded $J$-holomorphic $-1$-spheres which intersect one
another symplectically orthogonally,
\item $H^2(X,C;\RR)=0$.
\end{itemize}

\begin{dfn}
A \emph{standard configuration} in $X$ will mean a configuration $S=\bigcup_i S_i$ of
embedded symplectic spheres such that
\begin{itemize}
\item $[S_i]=[C_i]$ for all $i$,
\item there exists a $J\in\mJ$ simultaneously making every component $S_i$ into a $J$-holomorphic sphere,
\item at every intersection point of the configuration, the components intersect $\omega$-orthogonally.
\end{itemize}
Let $\mC_0$ denote the space of standard configurations.
\end{dfn}

\begin{prp}\label{delpcontr}
 $\mC_0$ is weakly contractible.
\end{prp}
\begin{proof}
The proof of this proposition is very similar to the proof of Proposition \ref{cstarccontr}. The Gompf isotopy argument is slightly more involved, as there are several components that need to be isotoped at once. The crucial input from Gromov's theory of pseudoholomorphic curves is the following:
\begin{lma}[\cite{Pin2}, Lemma 1.2]\label{grexist}
Let $(M, \omega)$ be a symplectic 4-manifold not diffeomorphic to $\PP{2}\#\overline{\PP{2}}$.
Then, for any choice of $\omega$-tame almost complex structure $J$, all symplectic exceptional classes of minimal symplectic area are represented by an embedded $J$-holomorphic
sphere.
\end{lma}
\end{proof}
In particular, any $S_1,S_2\in\mC_0$ are isotopic through standard configurations. The property that the configurations are symplectically orthogonal where they intersect and the condition $H^2(X,C;\RR)=0$ allow us to extend such an isotopy to a global symplectomorphism of $X$ by Banyaga's theorem. Therefore:
\begin{prp}
$\Symp_0(X)$ acts transitively on $\mC_0$.
\end{prp}

Banyaga's theorem now gives us a fibration:
\[
\begin{CD}
\Stab(C) @>>> \Symp_0(X) \\
@. @VVV \\
@. \mC_0
\end{CD}
\]
so the group $\Stab(C)$ of symplectomorphisms of $X$ which fix the
configuration $C$ setwise (and act trivially on the set of components) is
weakly homotopy equivalent to $\Symp_0(X)$ (by the homotopy long exact
sequence, since $\mC_0$ is weakly contractible).

The next step is to investigate $\Stab(C)$. An element $\phi\in\Stab(C)$ restricts to give an element $\left(\phi|_{C_1},\ldots,\phi|_{C_n}\right)\in\Symp(C)$ (see section \ref{confC}). Working in a neighbourhood of $C$, pulling back Hamiltonians to the normal bundles and cutting them off outside a certain radius, one can prove homotopy-lifting for this restriction map (see for example the proof of Proposition \ref{psifib} later). Therefore the restriction map fits into a fibration:
\begin{equation}\tag{*}
\begin{CD}
\Stab^0(C) @>>> \Stab(C) \\
@. @VVV \\
@. \Symp(C)
\end{CD}
\end{equation}
where $\Stab^0(C)$ is the group of symplectomorphisms of $X$ which fix $C$
pointwise.

As explained in section \ref{sympC}, the group $\Symp(C)$ can be understood purely in terms of $\Diff^+(C)$, thanks to Moser's theorem, and thereby reduced to a problem in Teichm\"{u}ller theory. To understand the stabiliser $\Stab^0(C)$, we need to consider the fibration:
\[
\begin{CD}
\Stab^1(C) @>>> \Stab^0(C) \\
@. @VVV \\
@. \mG
\end{CD}
\]
which takes a symplectomorphism fixing $C$ to the induced map on the normal bundles of $C$ (see section \ref{gau}). Here $\mG$ is the group of symplectic gauge transformations of the normal bundles to components of $C$ and $\Stab^1(C)$ consists of those symplectomorphisms fixing $C$ pointwise and acting trivially on the normal bundles of its components. By the symplectic neighbourhood theorem:
\begin{lma}
$\Stab^1(C)$ is weakly homotopy equivalent to the group of compactly supported
symplectomorphisms of $U=X\setminus C$.
\end{lma}
In all our cases we will understand $\Symp_c(U)$ and $\mG$. The homotopy exact
sequences of the various fibrations (pictured below) will allow us to work
backward, finding the homotopy groups of $\Stab^0(C)$ and
$\Stab(C)\simeq\Symp_0(X)$.
\[
\begin{CD}
\Symp_c(U) @>>> \Stab^0(C) @>>> \Stab(C) @>{\simeq}>> \Symp_0(X) \\
@. @VVV @VVV @VVV \\
@. \mG @. \Symp(C) @. \Orb(C)\simeq\star
\end{CD}
\]

\subsection{The case $\DD_3$}

Here we consider the configuration of exceptional curves in the homology
classes $S_{12}=H-E_1-E_2$, $S_{13}=H-E_1-E_3$, $S_{23}=H-E_2-E_3$, $E_1$ and
$E_2$. If we included $E_3$, these would form a hexagonal configuration, but in
order to ensure $H^2(X,C;\RR)=0$, we need to omit $E_3$.
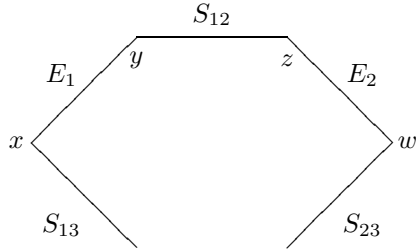
\[
\xy
(0,-3)*{y};(20,-3)*{z};
(36,-14)*{w};(-16,-14)*{x};
(0,0)*{};(20,0)*{} **\dir{-};
(10,3)*{S_{12}};
(-14,-14)*{};(0,0)*{} **\dir{-};
(-10,-5)*{E_1};
(34,-14)*{};(20,0)*{} **\dir{-};
(30,-5)*{E_2};
(-14,-14)*{};(0,-28)*{} **\dir{-};
(-10,-25)*{S_{13}};
(20,-28)*{};(34,-14)*{} **\dir{-};
(30,-25)*{S_{23}};
\endxy
\]
The complement of $U=X\setminus C$ is biholomorphic to $\CC\times\CC^*$ and
$\omega$ restricted to $U$ is $-dd^ch$ for a plurisubharmonic function $h$ on
$U$. By Proposition \ref{twostein}, $\Symp_c(U)\simeq\Symp_c(\CC\times\CC^*)$.
By Theorem \ref{cstarcthm}, $\Symp_c(\CC\times\CC^*)$ is
contractible, so $\Stab^0(C)\simeq\mG$. Let us calculate the groups $\mG$ and
$\Symp(C)$:
\begin{itemize}
\item $\mG(S_{13})\cong\mG(S_{23})\cong\mG_1\simeq\star$,
\item $\mD(S_{13})\cong\mD(S_{23})\cong\mD_1\simeq S^1$,
\item $\mG(S_{12})\cong\mG(E_1)\cong\mG(E_2)\cong\mG_2\simeq\ZZ$,
\item $\mD(S_{12})\cong\mD(E_1)\cong\mD(E_2)\cong\mD_2\simeq S^1$,
\end{itemize}
\noindent therefore $\mG\simeq\ZZ^3$ and $\Symp(C)\simeq (S^1)^5$. Since
$\Stab^0(C)\simeq\mG$ and $\Stab(C)\simeq\Symp(\DD_3)$, the fibration
(\ref{stab0fib}) yields the long exact sequence:
\[1\rightarrow\pi_1(\Symp(\DD_3))\rightarrow\ZZ^5\rightarrow\ZZ^3\rightarrow\pi_0(\Symp(\DD_3))\rightarrow
1\]
The calculation reduces to understanding the map $\ZZ^5\rightarrow\ZZ^3$. This
comes from a map $\pi_1(\Symp(C))\rightarrow\pi_0(\mG)$, so this can actually
be understood purely by considering a neighbourhood of $C$ as explained in section \ref{confC}. Let us pick a basis for $\pi_1(\Symp(C))$. Let $\rot(C_j)$ denote the element
represented by the loops $\phi_{2\pi}$ defined in Lemma \ref{localcomp}, i.e. the loop of sympletomorphisms which rotate $C_j$ through $2\pi$ and leave the other components fixed. These generate the group $\pi_1(\Symp(C))$.

By Lemma \ref{localcomp}, the map $\ZZ^5\rightarrow\ZZ^3$ is:
\[\begin{array}{cccc}
\rot(S_{13}) & \mapsto & g_{E_1}(x) & \in\pi_0(\mG(E_1)) \\
\rot(S_{23}) & \mapsto & g_{E_2}(w) & \in\pi_0(\mG(E_2))
\end{array}\]
and
\[\begin{array}{ccll}
\rot(E_1) & \mapsto & (0,g_{S_{13}}(x)) &
\in\pi_0(\mG(S_{13}))\times\pi_0(\mG(S_{12}) \\
\rot(E_2) & \mapsto & (g_{S_{12}}(z),0) &
\in\pi_0(\mG(S_{12}))\times\pi_0(\mG(S_{23})) \\
\rot(S_{12}) & \mapsto & (g_{E_1}(y),g_{E_2}(z)) &
\in\pi_0(\mG(E_1))\times\pi_0(\mG(E_2))
\end{array}\]
Therefore the map is surjective with kernel of rank 2. This implies
\[\pi_1(\Symp_0(\DD_3))=\ZZ^2,\ \pi_0(\Symp_0(\DD_3))=0\]
as stated.

\subsection{The case $\DD_4$}

The configuration $C$ consists of exceptional spheres in the homology classes
$S_{12}=H-E_1-E_2$, $S_{34}=H-E_3-E_4$, $E_1$, $E_2$, $E_3$ and $E_4$.
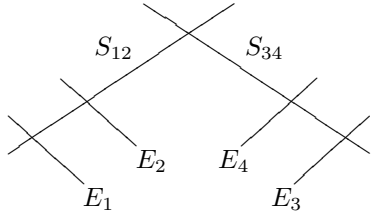
\[
\xy
(0,-20)*{};(30,0)*{} **\dir{-};
(18,0)*{};(48,-20)*{} **\dir{-};
(0,-15)*{};(10,-24)*{} **\dir{-};
(7,-10)*{};(17,-19)*{} **\dir{-};
(48,-15)*{};(38,-24)*{} **\dir{-};
(41,-10)*{};(31,-19)*{} **\dir{-};
(12,-26)*{E_1};(19,-21)*{E_2};
(37,-26)*{E_3};(30,-21)*{E_4};
(14,-6)*{S_{12}};(34,-6)*{S_{34}};
\endxy
\]
Let us calculate the homotopy types of $\mG$ and $\Symp(C)$:
\begin{itemize}
\item $\mG(S_{12})\cong\mG(S_{34})\cong\mG_3\simeq\ZZ^2$,
\item $\mD(S_{12})=\mD(S_{34})=\mD_3\simeq \star$,
\item $\mG(E_1)\cong\cdots\cong\mG(E_4)\cong\mG_1\simeq\star$,
\item $\mD(E_1)\cong\cdots\cong\mD(E_4)=\mD_1\simeq S^1$,
\end{itemize}
Therefore $\mG\simeq\ZZ^4$ and $\Symp(C)\simeq (S^1)^4$.

The complement $U=X\setminus C$ is biholomorphic to $\CC^2$ and the restriction
of $\omega$ to $U$ arises as $-dd^ch$ for a plurisubharmonic function $h$ on
$U$, so by Proposition \ref{twostein} $\Symp_c(U)\simeq\Symp_c(\CC^2)$. Gromov
has shown $\Symp_c(\CC^2)\simeq\star$. This simplifies the fibrations:
\[
\begin{CD}
\star @>>> \Stab^0(C) @>>> \Stab(C)\simeq\Symp_0(\DD_4) \\
@. @VV{\simeq}V @VVV \\
@.\mG\simeq\ZZ^4 @. \Symp(C)\simeq(S^1)^4
\end{CD}
\]
We therefore get a long exact sequences of homotopy groups:
\[1\rightarrow\pi_1(\Symp_0(\DD_4))\rightarrow\ZZ^4\rightarrow\ZZ^4\rightarrow\pi_0(\Symp_0(\DD_4))\rightarrow
1\]
In the notation introduced in the previous section, the map
$\ZZ^4\rightarrow\ZZ^4$ is given by:
\[\rot(E_i,x_i)\mapsto g_{S_{ik}}(x_i)\in\pi_0(\mG(S_{ik}))\]
so it is an isomorphism and all the homotopy groups of $\Symp_0(\DD_4)$ vanish.

\subsection{The case $\DD_5$}

The relevant configuration $C$ is the total transform of the conic through the
five blow-up points, i.e. the exceptional curves in homology classes
$E_1,\ldots,E_5$ and $Q=2H-\sum_{i=1}^5E_i$.
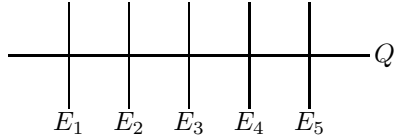
\[
\xy
(0,0)*{};(48,0)*{} **\dir{-};(50,0)*{Q};
(8,7)*{};(8,-7)*{} **\dir{-};(8,-9)*{E_1};
(16,7)*{};(16,-7)*{} **\dir{-};(16,-9)*{E_2};
(24,7)*{};(24,-7)*{} **\dir{-};(24,-9)*{E_3};
(32,7)*{};(32,-7)*{} **\dir{-};(32,-9)*{E_4};
(40,7)*{};(40,-7)*{} **\dir{-};(40,-9)*{E_5};
\endxy
\]
Let us calculate the homotopy types of $\mG$ and $\Symp(C)$:
\begin{itemize}
\item $\mG(Q)\cong\mG_5\simeq\ZZ^4$,
\item $\mD(Q)\cong\Diff(5,S^2)$
\item $\mG(E_1)\cong\cdots\cong\mG(E_5)\cong\mG_1\simeq\star$,
\item $\mD(E_1)\cong\cdots\cong\mD(E_5)\cong\mD_1\simeq S^1$,
\end{itemize}
so $\mG\simeq\ZZ^4$ and $\Symp(C)\simeq\Diff(5,S^2)\times(S^1)^5$.

The complement $U=X\setminus C$ is biholomorphic to the complement of a conic
in $\PP{2}$ and the restriction of $\omega$ to $U$ is $-dd^ch$ for a
plurisubharmonic function $h$, so by Proposition \ref{twostein}
$\Symp_c(U)\simeq\Symp_c(T^*\RP{2})$. By section \ref{rp2},
$\Symp_c(T^*\RP{2})\simeq\ZZ$. The fibrations become:
\[
\begin{CD}
\ZZ @>>> \Stab^0(C) @>>> \Stab(C)\simeq\Symp(\DD_5) \\
@. @VV{\simeq}V @VVV \\
@.\mG\simeq\ZZ^4 @. \Symp(C)\simeq\Diff(5,S^2)\times(S^1)^5
\end{CD}
\]
From this we deduce that $\Stab^0(C)$ is weakly equivalent to a discrete group $Z$ which is an extension of $\ZZ^4$ by $\ZZ$ and that
the relevant long exact sequence of homotopy groups is:
\[1\rightarrow\pi_1(\Symp(\DD_5))\rightarrow \ZZ^5\stackrel{\psi}{\rightarrow}
Z\rightarrow\pi_0(\Symp(\DD_5))\rightarrow B\rightarrow 1\]
where $B$ is the group $P\Br(5,S^2)/\langle\tau\rangle$ of components of
$\Diff(5,S^2)$.

The key, then, is to understand this map $\psi:\ZZ^5\rightarrow Z$. The
composition of this map with $\pi_0(\Stab^0(C))\rightarrow\mG\cong\ZZ^4$ is
surjective by Lemma \ref{localcomp} (the preimage of
$g_{Q}(x_i)\in\pi_0(\mG(Q))$ is $\rot(E_i,x_i)$). It remains to show that if
$\chi$ is the isotopy class of the Dehn twist in the Lagrangian $\RP{2}\subset
U$ then $\chi=\psi(\eta)$ for some $\eta\in\pi_1(\Symp(C))$.

We begin with a lemma whose proof is analogous to that of Lemma 1.8 from
\cite{Sei08} (which is the same statement for Lagrangian spheres):

\begin{lma}
Let $L\subset X$ be a Lagrangian $\RP{2}$, $\mho$ a Weinstein neighbourhood of
$L$ and suppose there is a Hamiltonian circle action on $X\setminus L$ which
commutes with the round cogeodesic flow on $\mho\setminus L$. Then the Dehn
twist in $L$ is symplectically isotopic to $\id$.
\end{lma}

Let $\mu$ be the moment map for the $SO(3)$-action on $T^*\RP{2}$. Then
$||\mu||$ generates a Hamiltonian circle action on $T^*\RP{2}\setminus\RP{2}$
which commutes with the round cogeodesic flow. Symplectically cutting along a
level set of $||\mu||$ gives $\PP{2}$ and the reduced locus is a conic. Pick
five points on the conic and $||\mu||$-equivariant balls of equal volume
centred on them. $\DD_5$ is symplectomorphic to the blow up in these five balls
and the circle action preserves the exceptional locus (the union of the five
blow-up spheres and the proper transform of the conic). Hence we have
constructed a loop in $\Symp(C)$ which maps to $\chi$. Since the circle action
simultaneously rotates the normal bundles to the blow-up spheres it should be
clear that this element corresponds to the diagonal element
$(1,\ldots,1)\in\ZZ^5$.

This completes the proof of Theorem \ref{sympgrps}.

\section{The $A_n$-Milnor fibres}\label{milnorsect}

In this section we prove the following theorem:
\setcounter{unthm}{3}
\begin{thmun}
Let $W$ be the $A_n$-Milnor fibre, the affine variety given by the equation
\[x^2+y^2+z^n=1,\]
and let $\omega$ be the K\"{a}hler form on $W$ induced from the ambient
$\CC^3$. Then the group of compactly supported symplectomorphisms of
$(W,\omega)$ is weakly homotopy equivalent to its group of components. This group of components injects homomorphically into the braid group $\Br_{n}$ of $n$-strands on the disc.
\end{thmun}

\subsection{Compactification}

We begin by
compactifying $W$ to a projective rational surface. Let $\xi_k=\exp(2\pi i k/n)$ and denote by $P_k(c)$ the polynomial $\frac{c^n-1}{c-\xi_k}$. Define
\begin{itemize}
\item $X$ to be the blow-up of $\PP{2}$ at the points
$\left\{[\xi_k:0:1]\right\}_{k=1}^{n}$. This can be thought of as the
subvariety of $M=\PP{2}\times\prod_{k=1}^{n}\PP{1}_k$ given by the equations
\[a_ky=b_k(x-\xi_kz)\mbox{ for }k=1,\ldots,n\]
in coordinates
\[\left([x:y:z],[a_1:b_1],\ldots,[a_{n}:b_{n}]\right)\mbox{ on }M.\]
\item the pencil of curves $P_t$ which are proper transforms of the pencil of lines through $[0:1:0]$, parametrised by $t=[x:z]$. In particular, notice that
\[P_{\infty}=\left\{\left([x:y:0],[x:y],\ldots,[x:y]\right)\right\}\]
\item the curve $C_{n+1}=P_{\infty}$.
\item the curve $C_{n+2}=\{[x:0:z]\}\times\prod_{k=1}^{n}\{[1:0]\}$ which is the proper transform of the line through the blow-up points. This is a section of the pencil $P_t$.
\item the exceptional spheres $C_k$ ($1\leq k\leq n$) of the blow-up. $C_k$ is given in equations by
\[\left\{[\xi_k:0:1]\right\}\times\{[1:0]\}\times\ldots\times\PP{1}_k\times\ldots\times\left\{[1:0]\right\}\]
and constitutes one component of the singular curve $P_{\xi_k}$.
\item the divisor $C'=C_{n+1}\cup C_{n+2}$ and its complement $U'=X\setminus C'$.
\item the divisor $C=C'\cup\bigcup_{i=1}^nC_k$ and its
complement $U=X\setminus C$.
\end{itemize}
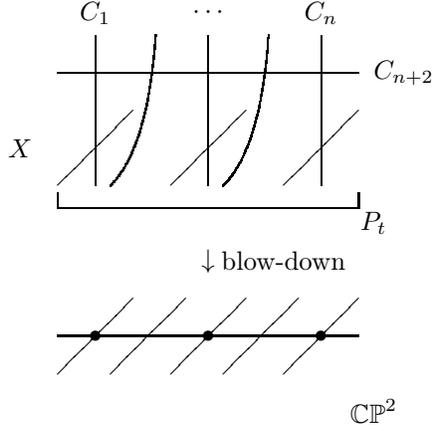
\[
\xy
(0,0)*{};(40,0)*{} **\dir{-};
(46,0)*{C_{n+2}};
(5,5)*{};(5,-15)*{} **\dir{-};(5,8)*{C_1};
(20,5)*{};(20,-15)*{} **\dir{-};(20,8)*{\cdots};
(35,5)*{};(35,-15)*{} **\dir{-};(35,8)*{C_n};
(10,-5)*{};(0,-15)*{} **\dir{-};
(25,-5)*{};(15,-15)*{} **\dir{-};
(40,-5)*{};(30,-15)*{} **\dir{-};(-5,-10)*{X};
(13,5)*{};(7,-15)*{} **\crv{(12,-10)};
(28,5)*{};(22,-15)*{} **\crv{(27,-10)};
(0,-18)*{};(40,-18)*{} **\dir{-};
(42,-20)*{P_t};
(0,-18)*{};(0,-16)*{} **\dir{-};
(40,-18)*{};(40,-16)*{} **\dir{-};
(20,-25)*{\downarrow};(30,-25)*{\mbox{blow-down}};
(0,-35)*{};(40,-35)*{} **\dir{-};
(10,-30)*{};(0,-40)*{} **\dir{-};(5,-35)*{\bullet};
(25,-30)*{};(15,-40)*{} **\dir{-};(20,-35)*{\bullet};
(40,-30)*{};(30,-40)*{} **\dir{-};(35,-35)*{\bullet};(42,-45)*{\PP{2}};
(17,-30)*{};(7,-40)*{} **\dir{-};
(32,-30)*{};(22,-40)*{} **\dir{-};
\endxy
\]
\begin{lma}\label{U}
$U'$ is biholomorphic to $W$ and $U$ is biholomorphic to $\CC^*\times\CC$.
\end{lma}
\begin{proof}
We proceed by introducing affine charts on $U'$ and writing down explicit
partial maps $W\bir U'$ which glue compatibly to give a biholomorphism.

Let $\CC_+$, $\CC_-$ denote the upper and lower (Zariski) hemispheres of $\PP{1}$
respectively. For any partition $\sigma$ of $\{1,\ldots,n\}$ into two sets
$\sigma_+$ and $\sigma_-$ let $(\pm)^{\sigma}(k)$ be $\pm 1$ if
$k\in\sigma_{\pm}$ respectively and consider the affine open set
\[M_{\sigma}=\CC^2\times\prod_{k=1}^{n}\CC_{(\pm)^{\sigma}(k)}\subset M\]
\noindent Denote by $U'_{\sigma}$ the intersection $U'\cap M_{\sigma}$. Notice
that unless $|\sigma_-|\leq 1$, $y$ cannot vanish, so
\[\bigcup_{|\sigma_-|\geq 2}U'_{\sigma}=\bigcap_{|\sigma_-|\geq
2}U'_{\sigma}=\left\{\left((x,y),[x-\xi_1:y],\ldots,[x-\xi_n:y]\right):y\neq
0\right\}\]
\noindent We define $\sigma^k$ to be the partition with $\sigma_-=\{k\}$. Since
we have excised $C_{n+2}$, we do not need to consider the partition
$\sigma^{\infty}$ with $\sigma_-=\emptyset$.

Recall that $W=\{(a,b,c)\in\CC^3 : a^2+b^2+c^n=1\}$. For each $\sigma$ we
define a partial map $\psi_{\sigma}:W\bir U'_{\sigma}$ (i.e. a map whose domain
is implicitly defined as the region on which the map is well-defined) by
\[\psi_{\sigma}(a,b,c)=\left((c,a+ib),[c-\xi_1:a+ib],\ldots,[c-\xi_n:a+ib]\right)\]
\noindent for $|\sigma_-|\geq 2$ and
\[\psi_{\sigma^k}(a,b,c)=\left((c,a+ib),[c-\xi_1:a+ib],\ldots,[a-ib:P_k(c)],\ldots[c-\xi_n:a+ib]\right)\]

\noindent Since $a^2+b^2=(a+ib)(a-ib)=1-c^n$ and $P_k(c)(c-\xi_k)=c^n-1$, the
standard $a_i\mapsto 1/a_i$ transition maps for the affine charts on $\PP{1}$
allow us to glue these partial maps into a global biholomorphism $W\rightarrow
U'$.

To identify $U$ and $\CC^*\times\CC$, notice that $U$ is just the subset
\[\left\{\left((x,y),[x-\xi:y],\ldots,[x-\xi^{n+1}:y]\right):y\neq 0\right\}\]
from above.
\end{proof}

We now construct a symplectic form on the compactification $X$. Let $\Omega$
denote the symplectic form on $M=\PP{2}\times\prod_{k=1}^n\PP{1}_k$ which is
the product of the Fubini-Study forms on each factor, normalised to give a
complex line area 1. $X$ inherits a symplectic form $\omega$ from its embedding
in $M$ and it is easy to see that
\begin{eqnarray*}
[C_{n+2}] & = & [C_{n+1}]-\sum_{k=1}^n[C_k] \\
\omega(C_{n+1}) & = & n+1 \\
\omega(C_{n+2}) & = & 1 \\
\omega(C_k) & = & 1 \\
P.D.[\omega] & = & n[C_{n+1}] +[C_{n+2}] \\
c_1(X) & = & 3[C_{n+1}]-\sum_{k=1}^n[C_k]
\end{eqnarray*}
Moreover, the following observation will prove extremely useful:

\begin{lma}\label{symporth}
At the intersection points $C_{n+2}\cap P_{\infty}$ and $C_{n+2}\cap C_k$, the various intersecting components are all $\omega$-symplectically orthogonal. Indeed $C_{n+2}$ is symplectically orthogonal to all the members of the pencil $P_t$.
\end{lma}

This can be checked by hand. We also notice that since $U$ is the intersection of an affine open subset of $M$ with $X$ that $\omega|_{U}$ arises from a plurisubharmonic function. The divisor $nC_{n+1}+C_{n+2}$ satisfies the Nakai-Moishezon criterion for ampleness, so by the first proposition in \cite{GrHa}, chapter 1, section 2 there is also a plurisubharmonic function $\phi_{C'}$ on $U'$ for which
$-dd^c\phi_{C'}=\omega|_{U'}$. By Proposition \ref{twostein}, this implies:

\begin{lma}
$\Symp_c(U)\simeq\Symp_c(\CC^*\times\CC)$ and $\Symp_c(U')\simeq\Symp_c(W)$.
\end{lma}

\subsection{Proof of Theorem \ref{milnorsymp}}

We now proceed to the proof of Theorem \ref{milnorsymp}. As before we introduce the notion of a standard configuration:

\begin{dfn}
A \emph{standard configuration} in $X$ will mean an \emph{unordered} $n$-tuple of embedded symplectic spheres $\{S_i\}_{i=1}^n$ satisfying the following properties:
\begin{itemize}
\item each $S_i$ is disjoint from $C_{n+1}$,
\item $[S_i]=[C_i]$,
\item there exists a $J\in\mJ$ for which all the spheres $S_i$, $C_{n+1}$ and $C_{n+2}$ are $J$-holomorphic.
\item there is a neighbourhood $\nu$ of $C_{n+2}$ such that, for all $i\in\{1,\ldots,n\}$,
\[S_i\cap\nu=P_{t_i}\cap\nu\]
where $t_i=S_i\cap C_{n+2}$.
\end{itemize}
Let $\mC_0$ denote the space of all standard configurations.
\end{dfn}

Note that we really do want unordered configurations here in order to obtain the full (rather than the pure) braid group.

\begin{prp}\label{milnorcontr}
$\mC_0$ is weakly contractible.
\end{prp}
\begin{proof}
The proof of this is similar to the proof of Proposition \ref{cstarccontr}. The input from Gromov's theory of pseudoholomorphic curves is Lemma \ref{grexist}, mentioned in the proof of Proposition \ref{delpcontr}.
\end{proof}

Let $\Conf(n)$ denote the space of configurations of unordered points in
$C_{n+2}\setminus\{\infty\}$, where $\infty$ denotes the point of intersection between $C_{n+2}$ and $P_{\infty}$. Define the map
\begin{align*}
\Psi:\mC_0&\rightarrow\Conf(n) \\
S=\bigcup_{i=1}^nS_i &\mapsto\{S_1\cap C_{n+2},\ldots,S_n\cap C_{n+2}\}.
\end{align*}

\begin{prp}\label{psifib}
$\Psi$ is a fibration.
\end{prp}
\begin{proof}
Let $Y$ be a test space and
\[
\begin{CD}
Y @>f>> \mC_0 \\
@VVV @VV{\Psi}V \\
Y\times I @>>h> \Conf(n)
\end{CD}
\]
\noindent be a commutative test diagram for the homotopy lifting property which defines a fibration. We must show that $h$ lifts to a map $\tilde{h}:Y\times I\rightarrow\mC_0$ and which commutes with the other maps.

Let us first notice that $h$ can be arbitrarily closely approximated by maps which are smooth in the $I$-direction (for example, by Theorem 2.6 of \cite{Hir76}, chapter 2). So if we can prove homotopy lifting for such maps then by taking limits of approximation sequences we have proved it for all maps. Let us therefore assume that $h$ is smooth in the $I$-direction.

For clarity of exposition let us also assume that $n=1$ and explain how the proof should be modified for $n>1$ at the end.

The idea behind constructing $\tilde{h}$ will be to find symplectomorphisms $\psi_{y,t}:X\rightarrow X$ such that:
\begin{itemize}
\item $\psi_{y,t}$ preserves $C'=C_{n+1}\cup C_{n+2}$ and fixes $C_{n+1}$ pointwise,
\item $\psi_{y,t}(S)\in\mC_0$ for any $S\in\mC_0$,
\item $\Psi(\psi_{y,t}(f(y)))=h(y,t)$.
\end{itemize}
The lift $\tilde{h}$ will then be defined by
\[\tilde{h}(y,t)=\psi_{y,t}(f(y))\]
For example, if $Y=\{y\}$ this is just path-lifting:
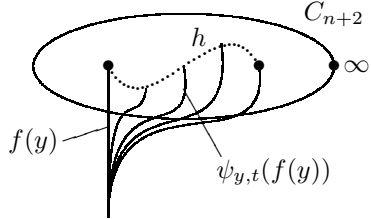
\[
\xy
(0,0)*\xycircle(20,7){-};(20,7)*{C_{n+2}};
(-10,0)*{\bullet};(10,0)*{\bullet} **\crv{~*=<2pt>{.}(-5,-10) &
(5,10)};(2,4)*{h};
(-10,0)*{};(-10,-20)*{} **\dir{-};(-20,-10)*{f(y)};
(-16,-10)*{};(-10,-8)*{} **\dir{-};
(-5,-3)*{};(-10,-20)*{} **\crv{(-6,-9) & (-10,0)};
(0,0)*{};(-10,-20)*{} **\crv{(2,-12) & (-10,0)};
(5,3)*{};(-10,-20)*{} **\crv{(7,-14) & (-10,0)};
(10,0)*{};(-10,-20)*{} **\crv{(12,-14) & (-10,0)};
(12,-14)*{\psi_{y,t}(f(y))};
(5,-12)*{};(0,-3)*{} **\dir{-};
(20,0)*{\bullet};(23,0)*{\infty};
\endxy
\]
Notice that $\Psi(\psi_{y,t}(f(y)))=\psi_{y,t}(\Psi(f(y)))$, so it suffices
that $\psi_{y,t}$ restricted to $C_{n+2}$ satisfies $\psi_{y,t}(\Psi(f(y)))=h(y,t)$.
In fact, we start by constructing the restriction $\bar{\psi}_{y,t}$ of
$\psi_{y,t}$ to $C_{n+2}$.

First, pick a continuous family of parametrised Darboux discs
$\{B_p\stackrel{\iota_p}{\rightarrow} C_{n+2}\setminus\{\infty\}\}$, one centred at each point $p\in C_{n+2}\setminus\{\infty\}$. For each $(y,t)\in Y\times I$ let $X(y,t)=\frac{\partial h(y,t)}{\partial t}$ and let $v(y,t)$ be the pullback of this vector to $B_{h(y,t)}$ along $\iota_{h(y,t)}$. Extend $v$ to a constant vector field on $B_{h(y,t)}$. This is generated by a linear Hamiltonian function which we can multiply by a cut-off function (equal to $1$ on a neighbourhood of $0\in B_{h(y,t)}$) to obtain a Hamiltonian which is compactly supported in the ball. Pulling back this Hamiltonian along $\iota_{h(y,t)}^{-1}$ gives a Hamiltonian $H_{y,t}$ on $C_{n+2}$, supported in a neighbourhood of $h(y,t)$. We call the time $t$ flow of this Hamiltonian $\bar{\psi}_{y,t}$. Notice that $\bar{\psi}_{y,t}(h(y,0))=h(y,t)$.

Let $N$ be a small tubular neighbourhood of $C_{n+2}$, $p:N\rightarrow C_{n+2}$ be the projection along the fibres of the pencil $P$ and let $\eta$ be a radial cut-off function on $N$ which equals 1 on a neighbourhood of $C_{n+2}$. The Hamiltonian function $\eta p^*H_{y,t}$ generates a time $t$ flow $\psi_{y,t}$ which extends $\bar{\psi}_{y,t}$ in such a way that it fixes $C_{n+1}$ and sends standard configurations to standard configurations.
To prove the result when $n>1$ the same construction is used in a Darboux chart near each point of $C_{n+2}\cap S$. This means that now we fix a (continuously varying) choice of $n$ disjoint Darboux discs over each configuration of $n$ distinct points in $C_{n+2}\setminus\{\infty\}$, modifying the proof accordingly.
\end{proof}

The space $\Conf(n)$ of configurations of $n$ points in the disc has
fundamental group $\Br_n$ and since $\mC_0$ is weakly contractible this group acts freely and transitively on the set of components of the fibre $\mF$ of $\Psi$. Moreover, $\Conf(n)$ is a $K(\Br_n,1)$-space, so the long exact sequence of homotopy groups associated to $\Psi$ implies:

\begin{lma}
$\mF\simeq\pi_0(\mF)$.
\end{lma}

We have an action of $\Symp_c(U')$ on the fibre
$\mF=\Psi^{-1}(\{\xi_1,\ldots,\xi_n\})$. Let us restrict attention to the orbit $\Orb(E)$ of the standard configuration of exceptional curves $C_1,\ldots, C_n$.

\begin{lma}
This consists of a union of connected components of $\mF$.
\end{lma}
\begin{proof}
If $S_0\in\Orb(E)\subset\mF$ and $S_1\in\mF$ lie in the same connected
component of $\mF$ then they are isotopic through standard configurations; let $S_t$ denote such an isotopy. Consider the isotopy $T_t=S_t\cup C'$ of configurations of symplectic spheres. By the assumption that $S_t$ is a standard configuration for each $t\in[0,1]$, this isotopy extends to an isotopy of neighbourhoods $\nu_t$ of $T_t$. We may also assume that this isotopy fixes a neighbourhood of $D=C_{n+2}\cup P_{\infty}$. The fact that $H^2(X,T_t;\RR)=0$ implies that this isotopy extends to a global symplectic isotopy of $X$ (by Banyaga's symplectic isotopy extension theorem). Hence there is a symplectomorphism of $X$ compactly supported on the complement of $C'$ (i.e. an element of $\Symp_c(U')$) which sends $S_0$ to $S_1$, so $\Orb(E)$ consists of a union of connected components of $\mF$.
\end{proof}

\begin{lma}
The stabiliser $\Stab(E)$ is weakly contractible.
\end{lma}
\begin{proof}
Arguing as in Lemma \ref{stabwkcontr}, we see that this is equivalent to
showing $\Symp_c(U)$ is weakly contractible. But we have already observed that
$\Symp_c(U)\simeq\Symp_c(\CC^*\times\CC)\simeq\star$.
\end{proof}

Banyaga's theorem gives a fibration:
\[
\begin{CD}
\Stab(E) @>>> \Symp_c(U') \\
@. @VVV \\
@. \Orb(E)
\end{CD}
\]
Weak contractibility of $\Stab(E)$ and of the components of $\Orb(E)$ implies
that $\pi_i(\Symp_c(U'))=0$ for $i>0$.

We finally prove that:
\begin{prp}
$\pi_0(\Symp_c(U'))$ injects homomorphically into the braid group $\Br_n$.
\end{prp}
\begin{proof}
The symplectomorphisms $\psi_{y,t}$ used to construct lifts in the proof of the fibration property of $\Psi$ are supported in some arbitrarily small neighbourhood of $C_{n+2}$. The action of $\Symp_c(U')$ is by definition supported on the complement of a neighbourhood of $C_{n+2}$. This implies that the action of $\pi_0(\Symp_c(U'))$ on $\pi_0(\mF)$ commutes with the action of $\Br_n$ coming from the fibration $\Psi$. When restricted to the components of $\Orb(E)\subset\mF$ we have seen that the $\pi_0(\Symp_c(U'))$ action is free. The $\Br_n$ action is free on the whole of $\pi_0(\mF)$. The claim that $\pi_0(\Symp_c(U'))$ injects into the braid group now follows from the following elementary lemma:

\begin{lma}
Let $G$ and $H$ be groups acting on a set $A$. Suppose $H$ acts freely and
transitively on $A$ and that $G$ acts freely on an orbit $\Orb_G(a)$ for some
$a\in A$. In particular, for each $b\in A$ there exists a unique $h_b$ such
that $h_b(a)=b$.  Define a map
\begin{align*}
f:G&\rightarrow H \\
f(\sigma)&=h_{\sigma(a)}^{-1}
\end{align*}
This map is injective. If the actions commute then it is a homomorphism.
\end{lma}
\end{proof}

\appendix
\section{Structures making a configuration holomorphic}\label{spaceofstructures}

Let $S=\bigcup_{i=1}^nS_i$ be a union of embedded symplectic 2-spheres in a symplectic 4-manifold $X$. Suppose that the various components intersect transversely and that there are no triple intersections. Suppose further that there is an $\omega$-compatible $J$ for which all the components $S_i$ are $J$-holomorphic. Let $\mH(S)$ denote the space of $\omega$-compatible almost complex structures $J$ for which all components $S_i$ are simultaneously $J$-holomorphic. We now provide a proof of Lemma \ref{contrH} from section \ref{confhol}.
\setcounter{unlma}{27}
\begin{lmaun}
If all the intersection between components of $S$ are symplectically orthogonal then the space $\mH(S)$ is weakly contractible.
\end{lmaun}
\begin{proof}
Let $\mJ(S_i)$ denote the (contractible) space of $\omega|_{S_i}$-compatible almost complex structures on $S_i$ and $\mJ(S)$ denote the product $\prod_{i=1}^n\mJ(S_i)$. The space $\mH$ maps to $\mJ(S)$ by restriction. We want to show this is a homotopy equivalence.

We proceed in stages. First define $\mJ|_{S_i}$ to be the space of $\omega$-compatible almost complex structures on $TX|_{S_i}$ and $\mJ|_S$ to be the product $\prod_{i=1}^n\mJ|_{S_i}$. The restriction map $\mH(S)\rightarrow\mJ(S)$ factors through restriction maps:
\[\mH(S)\rightarrow\mJ|_S\rightarrow\mJ(S)\]
We first analyse the map $\mJ|_S\rightarrow\mJ(S)$.
\begin{lma}\label{firstfibr}
$\mJ|_S\rightarrow\mJ(S)$ is a fibration.
\end{lma}
\begin{proof}
We illustrate the proof by proving path-lifting. Let $\nu_i$ denote the normal bundle to $S_i$, identified canonically as a subbundle of $TX|_{S_i}$ (the $\omega$-orthogonal complement to $TS_i$). Fix an $\omega|_{\nu_i}$-compatible almost complex structure $k_i$ on $\nu_i$. If
\[\gamma=(j_1(\cdot),\ldots,j_n(\cdot)):I\rightarrow\mJ(S)=\prod_{i=1}^n\mJ(S_i)\]
is a path of complex structures on $S=S_1\cup\cdots\cup S_n$ then $j_i\oplus k_i$ is an $\omega$-compatible almost complex structure on $TX|_{S_i}$. Unfortunately, there is no guarantee that $j_a(t)\oplus k_a=k_b\oplus j_b(t)$ at the intersection point $x\in S_a\cap S_b$. To rectify this, fix a Darboux ball $\iota:(B,\omega_0)\rightarrow (X,\omega)$ centred at $x$ for which $(S_a)$ is identified under $\iota$ with $\Pi_a=B\cap\RR^2\times\{0\}$ and $S_b$ with $\Pi_b=B\cap\{0\}\times\RR^2$. Let $\eta$ be a radial cut-off function on $B$ which is equal to $1$ on a neighbourhood of $0$ and equal to $0$ outside a small ball.

We now modify $k_a$ and $k_b$ so that they agree with $j_b(t)$ and $j_a(t)$ at $x$. Let $c$ stand for either $a$ or $b$. Let $g_c$ be the metric on the normal bundle $\nu\Pi_c$ to $\Pi_c$ corresponding to the pullback $\iota^*k_c$. Let $j(t)$ be the constant complex structure on $B$ such that $\iota_*j(t)=j_a(t)\oplus j_b(t)$ at $x$. This $j(t)$ defines a metric $\tilde{g}_c(t)$ on the normal bundle to $\Pi_c$. Since the space of metrics is convex, define a metric
\[g'_c(t)=(1-\eta)g_c+\eta \tilde{g}_c(t)\]
This in turn defines an endomorphism $A_c(t)$ of $\nu\Pi_c$ via $g'_c(t)=\omega_0(A_c(t)\cdot,\cdot)$. By the usual trick, $(A_c(t)A_c^{\dagger}(t))^{-1/2}A_c(t)$ defines an $\omega_0$-compatible almost complex structure on $\nu\Pi_c$. Push this forward along $\iota$. The new almost complex structures $k'_a(t)$ and $k'_b(t)$ agree with $k_a$ and $k_b$ outside the Darboux ball but now $k'_a(t)=j_b(t)$ and $k'_b(t)=j_a(t)$ at $x$.

Do this at every intersection point to get a 1-parameter family $k_i(t)$ of almost complex structures on $\nu_i$. These give almost complex structures $j_i(t)\oplus k_i(t)$ on $TX|_{S_i}$ such that $j_a(t)\oplus k_a(t)=k_b(t)\oplus j_b(t)$ at any intersection point $x\in S_a\cap S_b$ and hence we get a well-defined lift of $\gamma$ to a path in $\mJ|_S$.
\end{proof}

\begin{lma}
The fibre of $\mJ|_S\rightarrow\mJ(S)$ is contractible.
\end{lma}
\begin{proof}
The fibre is the space of almost complex structures on $TX|_S$ which are fixed along the tangent directions to each $S_i$. For each $S_a$, let $\mG(S_a)$ denote the group of symplectic gauge transformations of $\nu_a$ (the normal bundle to $S_a$) which equal the identity at intersection points $x\in S_a\cap S_b$. This acts transitively on the space of $\omega|_{\nu_a}$-compatible almost complex structures which agree with $j_a\oplus j_b$ at $x\in S_a\cap S_b$. Hence $\mG(S)=\prod_{i=1}^n\mG(S_i)$ acts transitively on the fibre of $\mJ|_S\rightarrow\mJ(S)$, making it into a homogeneous space $\mG(S)/\Stab(J)$. Since the inclusion $U(1)\rightarrow \mbox{Sp}(2)$ is a homotopy equivalence, $\mG(S)$ is homotopy equivalent to the group of unitary gauge transformations, $\mG_u(S)$, which is also the stabiliser of any given almost complex structure on the normal bundle. The homogeneous space $\mG(S)/\mG_u(S)$ is contractible.
\end{proof}

This shows that the space $\mJ|_S$ is weakly contractible, as it is the total space of a fibration over a contractible space with contractible fibres. We now analyse the restriction map $\mH(S)\rightarrow\mJ|_S$.

\begin{lma}
$\mH(S)\rightarrow\mJ|_S$ is a fibration.
\end{lma}
\begin{proof}
By the symplectic neighbourhood theorem, each component of $S_i$ has a neighbourhood $M_i$ isomorphic to a neighbourhood $\nu_i$ of the zero section in its normal bundle via an isomorphism $\phi_i:\nu_i\rightarrow M_i$. Since the various components intersect symplectically orthogonally, these identifications can be chosen compatibly in the sense that $S_j\cap M_i$ is identified with a normal fibre in $\nu_i$. Let $\gamma(t)=(J_1(t),\ldots,J_n(t))$ be a path in $\mJ|_S$. Each $\phi_i^*J_i(t)$ can be canonically extended to an $\omega$-compatible almost complex structure on $\nu_i$: $J_i(t)$ automatically defines an almost complex structure on the normal bundle to $S_i$, and the horizontal spaces $\omega$-orthogonal to the normal fibres give a connection which allows one to lift $J_i(t)|_{TS_i}$ to an almost complex structure on the total space of $\nu_i$.

This does not quite work at the intersection points where one must fix a Darboux chart and implant a local model. Suppose the two intersecting components are sent to the $\RR^2\times\{0\}$ and $\{0\}\times\RR^2$ planes in the Darboux chart and the almost complex structure $J(p,q)$ is specified along these planes (i.e. for points $(p,0)$ and $(0,q)$). Then, working with associated metrics, one can interpolate linearly as follows. Let $S^1(r)\times\{0\}$ and $\{0\}\times S^1(r)$ denote the unit circles (for the standard Euclidean metric) of the coordinate planes. For any point $x$ of $\RR^4$ there exists an $r$ and a unique line segment connecting $S^1(r)\times\{0\}$ to $\{0\}\times S^1(r)$ containing $x$. The analogous picture in $\RR^2=\RR\times\RR$ is familiar, where the dots lie on the unit circles:
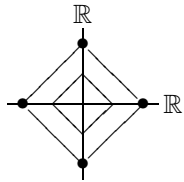
\[
\xy
(-10,0)*{};(10,0)*{} **\dir{-};(12,0)*{\RR};
(0,-10)*{};(0,10)*{} **\dir{-};(0,12)*{\RR};
(0,4)*{};(4,0)*{} **\dir{-};
(0,-4)*{};(4,0)*{} **\dir{-};
(0,-4)*{};(-4,0)*{} **\dir{-};
(0,4)*{};(-4,0)*{} **\dir{-};
(0,8)*{\bullet};(8,0)*{\bullet} **\dir{-};
(0,-8)*{\bullet};(8,0)*{} **\dir{-};
(0,-8)*{};(-8,0)*{\bullet} **\dir{-};
(0,8)*{};(-8,0)*{} **\dir{-};
\endxy
\]
Use the linear coordinate of this line to interpolate the associated metrics. Modifying this near the boundary of the Darboux ball to agree with the $J$ constructed from connections, we obtain an $\omega$-compatible almost complex structure $\tilde{\gamma}(t)$ on a neighbourhood of $S$.

Fix an arbitrary $\omega$-compatible almost complex structure $K$ on $X$. Interpolating associated metrics using a cut-off function in a neighbourhood of $S$, one can extend $\tilde{\gamma}(t)$ over $X$ so that it agrees with $K$ outside a neighbourhood of $S$. This gives a lift of the path $\gamma$. The proof of the covering homotopy property is similar but cumbersome.
\end{proof}

\begin{lma}
The fibre of $\mH(S)\rightarrow\mJ|_S$ is contractible.
\end{lma}
\begin{proof}
This is standard. The fibre $\mF$ consists of $\omega$-compatible almost complex structures making $S$ holomorphic which agree with a fixed almost complex structure $J$ along $S$. We define a deformation retraction of $\mF$ to $\{J\}$: for $J'\in\mF$, define the metric $g_{J'}=\omega(\cdot,J'\cdot)$. The path $g_{J'}(t)=tg_J+(1-t)g_{J'}$ of metrics defines a path of almost complex structures connecting $J'$ and $J$ in the usual way.
\end{proof}

This shows that $\mH(S)$ is weakly contractible, since it is the total space of a fibration over a weakly contractible space $\mJ|_S$ with contractible fibres.
\end{proof}

\end{document}